\documentclass{amsart}
\usepackage{amsmath, amsthm, amssymb,hyperref, graphicx, tikz, cancel}
\usepackage{mathtools}
\usepackage[margin=1in]{geometry}
 \usetikzlibrary{shapes.misc}
\usepackage{algorithm}
\usepackage{algpseudocode}
\usepackage[dvipsnames]{xcolor}

\usetikzlibrary{shapes.geometric}
\usetikzlibrary{shapes.symbols}
\usepackage{fontawesome5}

\newcommand{\SecondEulerian}[2]{\left\langle\!\!\! 
     \genfrac{\langle}{\rangle}{0pt}{}{#1}{#2} \!\!\!\right\rangle}

 \usepackage{pgfplots}
\usepackage{ytableau}
\usetikzlibrary{calc}

\definecolor{LMUCrimson}{RGB}{171, 12, 47}    
\definecolor{LMUBlue}{RGB}{16, 127, 184}    

\tikzset{
  pics/car/.style={
    code={
      \begin{scope}[scale=0.3]
        \draw[fill=LMUBlue!30, thick]
          (0,0) rectangle (4,1.5);

        \draw[fill=LMUBlue!30, thick]
          (1,1.5) -- (1.5,2.3) -- (2.8,2.3) -- (3.3,1.5) -- cycle;

        \fill (1,0) circle (0.35);
        \fill (3,0) circle (0.35);

        \node at (2,0.75) {#1};
      \end{scope}
    }
  }
}

\pgfplotsset{compat=1.17}
 \newtheorem{theorem}{Theorem}[section]
\newtheorem{lemma}[theorem]{Lemma}
\newtheorem{proposition}[theorem]{Proposition}
\newtheorem{corollary}[theorem]{Corollary}

\newtheorem{question}[theorem]{Question}

\theoremstyle{definition}
\newtheorem{definition}[theorem]{Definition}
\newtheorem{example}[theorem]{Example}

\theoremstyle{remark}
\newtheorem{remark}[theorem]{Remark}

\newcommand{\thefontsize}{The current size is: \f@size pt}

\title{Parking with Frustrated Drivers}

\author{Joshua Hallam}
\address{Department of Mathematics, Loyola Marymount University, Los Angeles, CA 90045 USA}
 \email{joshua.hallam@lmu.edu}
 
 \author{Jenson Molebash}
\address{Minneapolis, MN}
 \email{jensonmolebash@gmail.com}

\author{Chris Porter}
\address{Department of Mathematics, University of California Davis, Davis, CA 95616 USA}
 \email{chaporter@ucdavis.edu}
 


\begin{document}

 \begin{abstract}
    Imagine there are $n$ cars lined up along a one-way street containing $n$ spots. Each car contains a group of friends, including a reluctant driver.  Each car has a preferred spot and cars enter one by one. The cars drive to their preferred spot and if it is empty park there. If it is not empty, a friend in the back yells out ``Hey! You should have driven faster!". Frustrated by this, the driver continues down the road until they find the last unoccupied spot (if one exists) and parks there.  We say a sequence $(a_1,a_2,\dots, a_n)$ of preferred spots is a \textit{frustrated parking function} if all cars can park under this rule.

    In this paper, we study the enumerative properties of frustrated parking functions.
    In particular, we show that the number of frustrated parking functions of length $n$ is $(2n-1)!!$.   This is done  by associating frustrated parking functions with height labeled Dyck paths. Using this association, we are then able to better understand the sets of lucky cars and lucky spots for frustrated parking functions. We show  that the frustrated parking functions of length $n$ where the first $k$ cars (or first $k$ spots) are lucky is given by $k!S(n,k)$ where $S(n,k)$ is the Stirling number of the second kind. This in turn implies that the number of frustrated parking functions where  once a car (or spot) is unlucky, the remaining cars (or spots) are unlucky is counted by the $n^{th}$ Fubini number. We also show that the number of frustrated parking functions with $k$ lucky cars (or spots) is given by $\displaystyle\SecondEulerian{n}{k}$, the second order Eulerian number.
\end{abstract}

\maketitle

\section{Background}
Imagine you have  a one-way street with parking spots labeled by $[n]:=\{1,2,\dots, n\}$.   There are $n$ cars lined up to park and each has a preferred spot. One by one, the cars enter the lot and drive to their preferred spot. If it is empty, the car parks there. If not, the car drives to the next unoccupied spot (if it exists) and parks there. We say a sequence of preferences $(a_1,a_2,\dots,a_n)$ is a \textit{parking function} if cars  $1,2,\dots,n$ can all successfully park when car $i$ prefers parking spot $a_i$.  For example, the sequence $(1,1,3)$ is a parking function. On the other hand, $(1,4,3,3)$ is not a parking function since  there are three cars that prefer the spots $3$ or $4$, but only two such spots for them to park in.\\

Parking functions were originally explicitly defined by  Konheim and Weiss~\cite{konheim1966occupancy} who studied them in relation to hashing using linear probing. They showed that the number of parking functions of length $n$ is $(n+1)^{n-1}$.  Clearly, in a parking function, the number of cars that prefer spot $i$ or greater cannot exceed the number of such spots. That is, the number of cars with preference $i$ or larger cannot be more than $n-i+1$. Or, equivalently, the number of cars that prefer a spot no larger than $i$ is at least $i$.  What is less obvious is that this is a characterization of parking functions as was shown by Konheim and Weiss.\\

Since the introduction of parking functions, many connections have been made between parking functions and enumerative and algebraic combinatorics.  See Yan's survey~\cite{yan} for a detailed overview of the role parking functions play in combinatorics.  We also recommend Mart\'inez Mori's expository article~\cite{mori} for a brief introduction to the subject.\\

Over the past several years, there has been interest in understanding modifications to the parking rule. Such examples include allowing cars to back up from their preferred spot~\cite{Baumgardner2019,GenOfParkingFuncAllowBack}, only allowing cars to park within one spot of their preferred spot~\cite{Hadaway2021,unitIntRFubini,unitIntPFs}, only allowing cars to park in spots that form an interval~\cite{intervalPFs}, and allowing cars to have different sizes~\cite{permutationInvariantParkingFunctions}.   The  reader is encouraged to see~\cite{chooseOwnAdv} for a ``choose your own adventure" approach to learning about the many generalizations of parking functions.\\

Our work here focuses on considering a new modification to the parking procedure. In our modification, we still assume there are $n$ cars and $n$ spots and each enters the lot one at a time. Moreover, if the preferred spot is empty the car parks there. However, if the preferred spot is not available, the car drives to the \textit{furthest} available spot (as opposed to the nearest available spot).  As we mentioned in the abstract, we think of these drivers as being frustrated with their passengers and as  a result want to park as far away as possible from the preferred spot (while still respecting that  the road is one-way). We have thus named our sequences \textit{frustrated parking functions}. At times we will  use the terminology ``classical parking function" to differentiate them from frustrated parking functions.  \\

As  we will see in Proposition~\ref{prop:fpfsArePFs}, frustrated parking functions are indeed parking functions. However,  not all parking functions are frustrated parking functions.  For example, $(1,1,3)$ is a parking function, but not a frustrated parking function. To see why, lets consider how the cars park. First, we start with cars lined up at an empty lot as shown below.
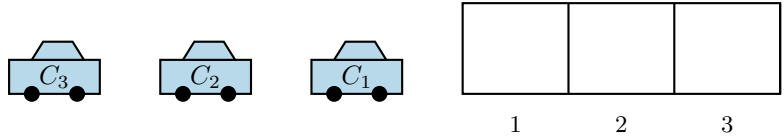
\begin{center}
    
  \begin{tikzpicture}[scale=1]

\def\spotwidth{1.4}
\def\spotheight{1.2}

\path (3*\spotwidth,0) rectangle ({4*\spotwidth},\spotheight);

\draw[thick] (0,0) rectangle ({3*\spotwidth},\spotheight);

\foreach \i in {1,...,3} {
    \draw[thick] ({\i*\spotwidth},0) -- ({\i*\spotwidth},\spotheight);
}

\foreach \i in {1,...,3} {
    \node at ({(\i-0.5)*\spotwidth},-0.4) {\small \i};
}

\pic at (-2,0) {car={$C_1$}};
\pic at (-4,0) {car={$C_2$}};
\pic at (-6,0) {car={$C_3$}};

\node at ({(1-0.5)*\spotwidth}, {0.5*\spotheight}) {$\textcolor{PineGreen}{}$ };

\node at ({(2-0.5)*\spotwidth}, {0.5*\spotheight}) {$\textcolor{BrickRed}{}$ };

\node at ({(3-0.5)*\spotwidth}, {0.5*\spotheight}) {$\textcolor{PineGreen}{}$ };

\end{tikzpicture}
\end{center}

Then car one enters. It prefers spot 1 and parks there.

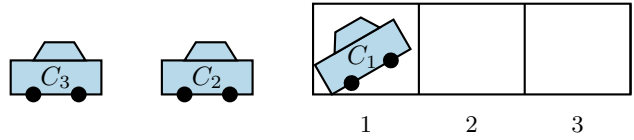
\begin{center}
    
  \begin{tikzpicture}[scale=1]

\def\spotwidth{1.4}
\def\spotheight{1.2}

\path (3*\spotwidth,0) rectangle ({4*\spotwidth},\spotheight);

\draw[thick] (0,0) rectangle ({3*\spotwidth},\spotheight);

\foreach \i in {1,...,3} {
    \draw[thick] ({\i*\spotwidth},0) -- ({\i*\spotwidth},\spotheight);
}

\foreach \i in {1,...,3} {
    \node at ({(\i-0.5)*\spotwidth},-0.4) {\small \i};
}

\pic[rotate=30] at (0.25,0) {car={$C_1$}};
\pic at (-2,0) {car={$C_2$}};
\pic at (-4,0) {car={$C_3$}};

\end{tikzpicture}
\end{center}

Then car two enters. It prefers spot 1, but finds it full. It goes to spot 3, and parks there.

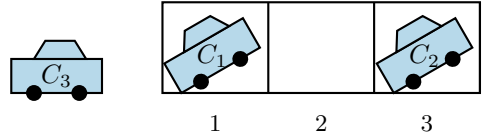
\begin{center}
    
  \begin{tikzpicture}[scale=1]

\def\spotwidth{1.4}
\def\spotheight{1.2}

\path (3*\spotwidth,0) rectangle ({4*\spotwidth},\spotheight);

\draw[thick] (0,0) rectangle ({3*\spotwidth},\spotheight);

\foreach \i in {1,...,3} {
    \draw[thick] ({\i*\spotwidth},0) -- ({\i*\spotwidth},\spotheight);
}

\foreach \i in {1,...,3} {
    \node at ({(\i-0.5)*\spotwidth},-0.4) {\small \i};
}

\pic[rotate=30] at (0.25,0) {car={$C_1$}};
\pic[rotate=30] at (3.06,0) {car={$C_2$}};
\pic at (-2,0) {car={$C_3$}};

\end{tikzpicture}
\end{center}

Now car three enters and wants to park in spot 3. However, it is full and therefore cannot park and so drives away.

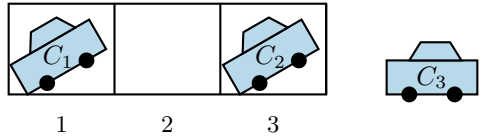
\begin{center}
    
  \begin{tikzpicture}[scale=1]

\def\spotwidth{1.4}
\def\spotheight{1.2}

\path (3*\spotwidth,0) rectangle ({4*\spotwidth},\spotheight);

\draw[thick] (0,0) rectangle ({3*\spotwidth},\spotheight);

\foreach \i in {1,...,3} {
    \draw[thick] ({\i*\spotwidth},0) -- ({\i*\spotwidth},\spotheight);
}

\foreach \i in {1,...,3} {
    \node at ({(\i-0.5)*\spotwidth},-0.4) {\small \i};
}

\pic[rotate=30] at (0.25,0) {car={$C_1$}};
\pic[rotate=30] at (3.06,0) {car={$C_2$}};
\pic at (5,0) {car={$C_3$}};
\end{tikzpicture}
\end{center}

 From this we see that $(1,1,3)$ is not a frustrated parking function. \\\

Now let's see an example  a frustrated parking function. Lets consider  the sequence (1,1,2). 
First, we start with cars lined up at an empty lot as shown below.

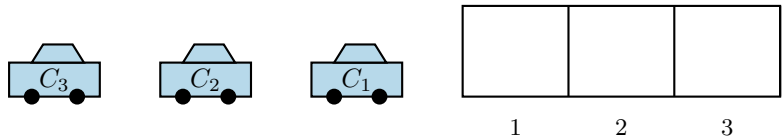
\begin{center}
  \begin{tikzpicture}[scale=1]
\def\spotwidth{1.4}
\def\spotheight{1.2}
\path (3*\spotwidth,0) rectangle ({4*\spotwidth},\spotheight);
\draw[thick] (0,0) rectangle ({3*\spotwidth},\spotheight);
\foreach \i in {1,...,3} {
    \draw[thick] ({\i*\spotwidth},0) -- ({\i*\spotwidth},\spotheight);
}
\foreach \i in {1,...,3} {
    \node at ({(\i-0.5)*\spotwidth},-0.4) {\small \i};
} 
\pic at (-2,0) {car={$C_1$}};
\pic at (-4,0) {car={$C_2$}};
\pic at (-6,0) {car={$C_3$}};
\node at ({(1-0.5)*\spotwidth}, {0.5*\spotheight}) {$\textcolor{PineGreen}{}$ };
\node at ({(2-0.5)*\spotwidth}, {0.5*\spotheight}) {$\textcolor{BrickRed}{}$ };
\node at ({(3-0.5)*\spotwidth}, {0.5*\spotheight}) {$\textcolor{PineGreen}{}$ };
\end{tikzpicture}
\end{center}

Then car one enters. It prefers spot 1 and parks there.

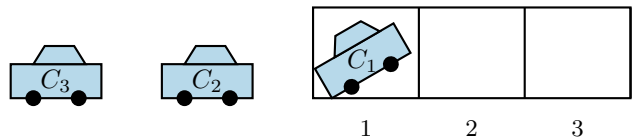
\begin{center}
    
  \begin{tikzpicture}[scale=1]

\def\spotwidth{1.4}
\def\spotheight{1.2}

\path (3*\spotwidth,0) rectangle ({4*\spotwidth},\spotheight);

\draw[thick] (0,0) rectangle ({3*\spotwidth},\spotheight);

\foreach \i in {1,...,3} {
    \draw[thick] ({\i*\spotwidth},0) -- ({\i*\spotwidth},\spotheight);
}

\foreach \i in {1,...,3} {
    \node at ({(\i-0.5)*\spotwidth},-0.4) {\small \i};
}

\pic[rotate=30] at (0.25,0) {car={$C_1$}};
\pic at (-2,0) {car={$C_2$}};
\pic at (-4,0) {car={$C_3$}};

\end{tikzpicture}
\end{center}

Then car two enters. It prefers spot 1 which is full and so parks at the furthermost spot, namely spot 3.  

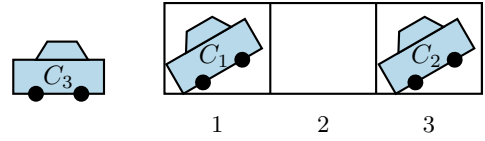
\begin{center}
    
  \begin{tikzpicture}[scale=1]

\def\spotwidth{1.4}
\def\spotheight{1.2}

\path (3*\spotwidth,0) rectangle ({4*\spotwidth},\spotheight);

\draw[thick] (0,0) rectangle ({3*\spotwidth},\spotheight);

\foreach \i in {1,...,3} {
    \draw[thick] ({\i*\spotwidth},0) -- ({\i*\spotwidth},\spotheight);
}

\foreach \i in {1,...,3} {
    \node at ({(\i-0.5)*\spotwidth},-0.4) {\small \i};
}

\pic[rotate=30] at (0.25,0) {car={$C_1$}};
\pic[rotate=30] at (3.06,0) {car={$C_2$}};
\pic at (-2,0) {car={$C_3$}};

\end{tikzpicture}
\end{center}

Now car three enters. It wants to park in spot 2. Since it is empty, it parks there. 

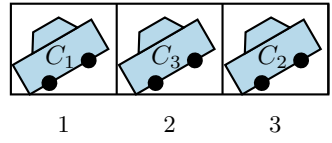
\begin{center}
    
  \begin{tikzpicture}[scale=1]

\def\spotwidth{1.4}
\def\spotheight{1.2}

\path (3*\spotwidth,0) rectangle ({4*\spotwidth},\spotheight);

\draw[thick] (0,0) rectangle ({3*\spotwidth},\spotheight);

\foreach \i in {1,...,3} {
    \draw[thick] ({\i*\spotwidth},0) -- ({\i*\spotwidth},\spotheight);
}

\foreach \i in {1,...,3} {
    \node at ({(\i-0.5)*\spotwidth},-0.4) {\small \i};
}

\pic[rotate=30] at (0.25,0) {car={$C_1$}};
\pic[rotate=30] at (3.06,0) {car={$C_2$}};
\pic[rotate=30] at (1.655,0) {car={$C_3$}};
\end{tikzpicture}
\end{center}

Let us also note that  even if a parking function is a frustrated parking function, that  does not mean that the end result of parking is the same. For example, consider $(1,1,2)$ again. As a classical parking function, the order of the spots that the cars park in is $1,2,3$. But as a frustrated parking function, the order is $1,3,2$.\\

The remainder of this paper is organized as follows. In the next section, we start by proving some basic facts about frustrated parking functions. We then count weakly increasing and weakly decreasing frustrated parking functions. Then in Section~\ref{sec:countingAll} we provide a map from frustrated parking functions to Dyck paths. Using this map, we show that the total number of frustrated parking functions of length $n$ is the same as the number of height labeled Dyck paths of semilength $n$ . This implies that there are  $(2n-1)!!$. Next in Section~\ref{sec:countLucky}, we use the map from frustrated parking functions to Dyck paths to understand lucky cars/spots. Briefly, a car is \textit{lucky} if it parks in its preferred spot and a spot is \textit{lucky} if the car that parked there is lucky. We show that, for a fixed set $T$, the number of frustrated parking functions of length $n$ whose set of lucky car positions is $T$ equals the  number of frustrated parking functions of length $n$ whose set of lucky spots is $T$. Note that this is something special about frustrated parking functions and the analogue for classical parking functions is not true. We then count frustrated parking functions where the first $k$ cars/spots are lucky and find a connection with unit interval parking functions.  In Section~\ref{sec:graphGeneralization},  we discuss how to generalize these ideas to a game on graphs. We finish by providing several open questions.\\

\begin{table}[H]
    \centering
   
\begin{tabular}{|c|c
|c|}\hline
   \textbf{  Family of Frustrated Parking Functions}& \textbf{Count/OEIS} & \textbf{Reference} \\
     \hline\hline
    All  & $(2n-1)!!$ \href{https://oeis.org/A001147}{(A001147)}& Theorem~\ref{thm:numTotalFPFsIsDoubleFact} \\
    \hline
    Weakly increasing  & $2^{n-1} $  \href{https://oeis.org/A000079}{(A000079)}& Corollary~\ref{cor:numInc}\\
    \hline
     Weakly decreasing  & $\displaystyle\dfrac{1}{n+1}{2n\choose n}$ \href{https://oeis.org/A000108}{(A000108)}  & Proposition~\ref{prop:weaklyDecPFAreFPFs}\\
     \hline
     $k$ lucky cars/spots & $\displaystyle\SecondEulerian{n}{k}$ \href{https://oeis.org/A008517}{(A008517)} & Corollary~\ref{cor:numWithKDescentsIsSecondEuler}\\
     \hline
     First $k$ cars/spots are lucky &  $k!S(n,k)$\href{https://oeis.org/A019538}{(A019538)} &  Proposition~\ref{prop:numFirstSpotsLuckyisk!S(n,k)}\\
     \hline
     Once a car/spot unlucky, remaining cars/spots  unlucky & $n^{th}$ Fubini number \href{https://oeis.org/A000670}{(A000670)} & Corollary~\ref{cor:numLuckyThenUnluckyIsFubini}\\
     \hline
\end{tabular}
\vspace{5 pt}

    \caption{Counts of Families of Frustrated Parking Functions of Length $n$}
    \label{tab:enumerativeResultTable}
\end{table}

The results in this paper are enumerative in nature. As such, in Table~\ref{tab:enumerativeResultTable}, we have provided a list of subsets of frustrated parking functions, their counts, and the reference in the paper for the result. Note that we use $\displaystyle\SecondEulerian{n}{k}$ for the second order Eulerian number indexed by $n$ and $k$ as well as $S(n,k)$ for the Stirling number of the second kind indexed by $n$ and $k$.

\section{Enumerative Results}
Before we prove our enumerative results, we need to prove some basic facts about frustrated parking functions.
\subsection{Basic Properties of Frustrated Parking Functions}
\begin{proposition}\label{prop:fpfsArePFs}
    If $(a_1,a_2,\dots, a_n)$ is a frustrated parking function, then  $(a_1,a_2,\dots, a_n)$ is also a parking function.
\end{proposition}
\begin{proof}
    We prove the contrapositive.  Suppose that  $(a_1,a_2,\dots, a_n)$  is not a parking function. Then there must be some $i$ such that the number of cars that prefer spot $i$ or greater is more than $n+1-i$.  But then when the cars park with respect to the frustrated parking rule, these   cars must still park in spots $i,i+1,\dots, n$. This cannot happen as there are more cars than spots. Hence  $(a_1,a_2,\dots, a_n)$ is not a frustrated parking function.
\end{proof}

As we have already seen with the case $(1,1,3)$, the converse of this statement is false.   It is not hard to see that both $(1,3,1)$ and $(3,1,1)$ are frustrated parking functions.  It follows that, unlike  parking functions, a rearrangement of  a frustrated parking function is not necessarily a frustrated parking function.  Nonetheless, we do have the following.

\begin{proposition}\label{prop:SwapProp}
    Let $(a_1,a_2,\dots, a_{i-1},a_i,a_{i+1},\dots, a_n)$ be a frustrated parking function. If $a_i<a_{i+1}$, then $(a_1,a_2,\dots, a_{i-1}, a_{i+1},a_i,\dots, a_n)$ is a frustrated parking function.
\end{proposition}

\begin{proof}
    Let $P=(a_1,a_2,\dots, a_{i-1},a_i,a_{i+1},\dots, a_n)$ and let $P'=(a_1,a_2,\dots, a_{i-1}, a_{i+1},a_i,\dots, a_n)$.  With respect to $P$ and $P'$,  the first $i-1$ cars park exactly the same and hence in $P'$ at least the first $i-1$ cars can park.  Suppose that for $P$, car $i$ parks in spot $c$ and car $i+1$ parks in spot $d$. We will show that the spots that are filled by the time the first $i+1$ cars park with respect to $P$ and $P'$ are the same. This will then imply $P'$ is a frustrated parking function. We break into cases depending on whether $c<d$ or $c>d$. \\

    \underline{Case 1}:  Suppose that $c<d$. Since car $i$ parks in spot $c$ with respect to $P$, we know that $a_i\leq c$. In fact, we claim that $a_i=c$. If not, then with respect to $P$, car $i$ does not park in its preferred spot. But then it would have parked in spot $d$ since $c<d$ and $d$ is open. So $a_i=c$. Since $a_i<a_{i+1}$, we have that $a_{i+1}>c$. It then follows that with respect to $P'$, when car $i$ enters and prefers spot $a_{i+1}$, it will consider the same spots as was the case in $P$. So it will park in spot $d$ with respect to $P'$. Then car $i+1$ comes in and prefers spot $a_i=c$. It is open so it parks there.\\

    \underline{Case 2}: Now let us assume that $d<c$. Note that since $a_{i+1}$ 
    parked in spot $d$, $a_{i+1}\leq d$.  Since $a_i<a_{i+1}$, we have that $a_i<d<c$.  Note this means that car $i$ did not park in its preferred spot with respect to  $P$.\\

 Suppose that $a_{i+1}=d$. Then in $P'$ car $i$ parks in spot $d$. Then car $i+1$ comes in preferring a spot before $d$. It finds it full as it did with respect to $P$ and then parks in spot $c$. \\

Now suppose that $a_{i+1}\neq d$. Then $a_{i+1}<d$. Since $c>d$, we know that $a_{i+1}\neq c$. Hence car $i$'s preferred spot must be filled with respect to $P'$  when car $i$ enters. It follows that it looks for the largest empty spot. This must have been spot $c$ since if there was a larger one, car $i$ in $P$ would have parked there. So car $i$ parks in spot $c$ with respect to $P'$. Now car $i+1$ comes in and prefers a spot before $d$. Its preferred spot must be filled (since in $P$ it  parked in spot $c$ and $a_i<c$). The largest empty spot is now $d$ since otherwise car $i+1$ in $P$ would have parked at a larger spot. So, car $i+1$ parks in spot $d$ with respect to $P'$.
 \end{proof}

\subsection{Counting Weakly Increasing and Weakly Decreasing Frustrated Parking Functions}

\begin{proposition}\label{prop:increasingSeqCharacterization}
    Suppose that $(a_1,a_2,\dots, a_n)$ is a weakly increasing sequence. Then $(a_1,a_2,\dots, a_n)$  is a frustrated parking function if and only if $\{a_1,a_2,\dots, a_n\}=[k]$ for some $k$.
\end{proposition}
\begin{proof}
    $(\Rightarrow)$ We prove the contrapositive. Suppose that $\{a_1,a_2,\dots, a_n\}\neq [k]$ for all $k$.   If $1\notin \{a_1,a_2,\dots, a_n\}$, then no car will park in spot 1, so $(a_1,a_2,\dots, a_n)$ is not a frustrated parking function. Otherwise, it must be the case that there is some $i$ such that
    $$
    i,i+r\in\{a_1,a_2,\dots, a_n\}
    $$ 
    with $r>1$ and
     $$
    i+(r-1) \notin\{a_1,a_2,\dots, a_n\}
    $$
    We claim in this case that spot $i+(r-1)$ will not get filled.  Since $(a_1,a_2,\dots, a_n)$ is weakly increasing there exists a
    $2\leq j\leq n$ such that $a_{j-1}=i$ and    $a_{j}=i+r$.  Furthermore, the fact that $(a_1,a_2,\dots, a_n)$ is weakly increasing implies that for spot $i+(r-1)$ to be filled, it would have to be filled by the time the $(j-1)^{st}$ car parks. This in turn implies  that all the spots from $i+(r-1)$ to $n$ are filled by the time $(j-1)^{st}$ car parks. But, $a_j=i+r$ and so the $j^{th}$ car will not be able to park.\\

    $(\Leftarrow)$ Now suppose that $\{a_1,a_2,\dots, a_n\}=[k]$ for some $k$. We claim that $(a_1,a_2,\dots, a_n)$ is a frustrated parking function. Suppose this was not the case.  Then at some point, one of the cars cannot park. Suppose that car $j$ cannot park and that $a_j=i$. We break into two cases depending on if $a_{j-1}=i$ or $a_{j-1}=i-1$. If $a_{j-1}=i$, then spot $i$ must be taken by the time car $j$ parks.  Since car $j$ cannot park, this means that all spots from $i+1$ to $n$ must have been filled earlier.  Since $(a_1,a_2,\dots, a_n)$ is weakly increasing, these must have been cars that did not park in their preferred spots.  It follows that in the first $j-1$ terms of the sequence, the preferences $1,2,\dots i$ appeared along with $(n+1)-(i+1)$ repeated preferences.  This means that $j-1=i+(n+1)-(i+1)=n$.  And so $j=n+1$ which is impossible as the string is only of length $n$.\\

    On the other hand, if $a_{j-1}=i-1$ and the $j^{th}$ car can't park, then by the time the $j-1$ car parks, spots $i$ to $n$ are filled with cars that did not park in their preferred spots. Once again this implies that $j-1=(i-1)+(n+1)-i=n$ , which is impossible.
\end{proof}

As the previous proposition shows, weakly increasing frustrated parking functions are just weakly increasing strings starting at $1$ that have no gaps between preferences. Consider the map that takes in such a string $(a_1,a_2,\dots, a_n)$ and returns the composition $(c_1,c_2,\dots, c_m)$ where $c_i=|\{j\mid a_j=i\}|$.  For example, $(1,1,2,2,2,3,3)\mapsto(2,3,2)$.  It is not hard to see that this map is a bijection between weakly increasing frustrated parking functions of length $n$ and compositions of $n$. Since there are $2^{n-1}$ such compositions, we get the following.

\begin{corollary}\label{cor:numInc}
    The number of weakly increasing frustrated parking functions of length $n$ is $2^{n-1
    }$.
\end{corollary}

Combining Proposition~\ref{prop:SwapProp}  and Proposition~\ref{prop:increasingSeqCharacterization}  immediately gives us the following.

\begin{corollary}\label{cor:ifLuckyCarIsIntervalItWorks}
   If $\{a_1,a_2,\dots, a_n\}=[k]$ for some $k$, then $(a_1,a_2,\dots, a_n)$  is a frustrated parking function.
\end{corollary}

As we previously saw with the parking function $(1,1,3)$, not every weakly increasing parking function is a frustrated parking function. However, it is the case that the weakly decreasing parking functions are frustrated  parking functions.  Intuitively this is because it is easier to park with the frustrated parking rules when the larger preferences occur early in the sequence.

\begin{proposition}\label{prop:weaklyDecPFAreFPFs}
    If $(a_1,a_2,\dots, a_n)$ is a weakly decreasing parking function, then it is a frustrated parking function. Consequently, there are $\displaystyle \dfrac{1}{n+1}{2n\choose n}$ weakly decreasing frustrated parking functions.
\end{proposition}
\begin{proof}
    Suppose this was not the case. Then there is some first car that cannot park. Suppose this car prefers spot $i$.  Then spots $i,i+1,\dots, n$ must all be filled. Since the sequence is weakly decreasing, every car that came before this one prefers a spot in $\{i,i+1,\dots, n\}$. But this means more cars prefer a spot in  $\{i,i+1,\dots, n\}$ than there are elements in  $\{i,i+1,\dots, n\}$ implying $(a_1,a_2,\dots, a_n)$ is not a parking function, a contradiction.
\end{proof}

\subsection{Counting all Frustrated Parking Functions}\label{sec:countingAll}
In this section, we prove that the number of frustrated parking functions of length $n$ is $(2n-1)!!$. We start by describing a way to associate a Dyck path to a frustrated parking function. The Dyck path will encode how the cars of the frustrated parking function behave. Dyck paths consist of unit up steps in the direction $(1,1)$ and unit down steps in the direction $(1,-1)$.  We will think of  Dyck paths  as strings consisting of the characters $U$ and $D$ corresponding to up steps and down steps respectively. For example, the Dyck path shown in Figure~\ref{fig:DyckPath} is the string $UUUDDDUUDUDDUUDD$. Given a Dyck path $D$, the \textit{semilength} of $D$ is the number of up steps (or down steps) in $D$. For example, the Dyck path shown in Figure~\ref{fig:DyckPath} has semilength $8$.\\

Before we describe this map, we need to define lucky cars and lucky spots.

\begin{definition}
    Let $(a_1,a_2,\dots, a_n)$ be a frustrated parking function. We say a car with preference $i$ is \textit{lucky} if it parks in spot $i$, otherwise we say it is \textit{unlucky}. Moreover, we say the spot $i$ is \textit{lucky} if the car that parks in spot $i$ is lucky. Otherwise, we call the spot \textit{unlucky}.
\end{definition}

\begin{example}
   Consider the frustrated parking function $(5,1,5,3,1,1)$.  Figure~\ref{fig:outcomeOfFPF} shows the outcome of parking with these preferences. Lucky cars are denoted in \textcolor{PineGreen}{green} and unlucky cars are denoted in \textcolor{BrickRed}{red}.
\end{example}

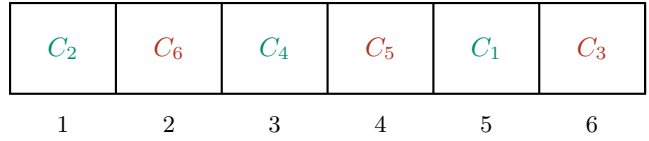
\begin{figure}

  \begin{tikzpicture}[scale=1]

\def\spotwidth{1.4}
\def\spotheight{1.2}

\draw[thick] (0,0) rectangle ({6*\spotwidth},\spotheight);

\foreach \i in {1,...,6} {
    \draw[thick] ({\i*\spotwidth},0) -- ({\i*\spotwidth},\spotheight);
}

\foreach \i in {1,...,6} {
    \node at ({(\i-0.5)*\spotwidth},-0.4) {\small \i};
}
\node at ({(1-0.5)*\spotwidth}, {0.5*\spotheight}) {$\textcolor{PineGreen}{C_2}$ };

\node at ({(2-0.5)*\spotwidth}, {0.5*\spotheight}) {$\textcolor{BrickRed}{C_6}$ };

\node at ({(3-0.5)*\spotwidth}, {0.5*\spotheight}) {$\textcolor{PineGreen}{C_4}$ };

\node at ({(4-0.5)*\spotwidth}, {0.5*\spotheight}) {$\textcolor{BrickRed}{C_5}$ };

\node at ({(5-0.5)*\spotwidth}, {0.5*\spotheight}) {$\textcolor{PineGreen}{C_1}$ };

\node at ({(6-0.5)*\spotwidth}, {0.5*\spotheight}) {$\textcolor{BrickRed}{C_3}$ };
\end{tikzpicture}
    \caption{Outcome of the frustrated parking function $(5,1,5,3,1,1)$. Lucky cars appear in \textcolor{PineGreen}{green} and unlucky cars appear in \textcolor{BrickRed}{red}. }
    \label{fig:outcomeOfFPF}
\end{figure}

From the figure, we see that the first, second, and fourth cars are lucky.  On the other hand, the first, third, and fourth spots are lucky.  We wish to point out that while clearly the number of lucky spots and lucky cars must always be the same, the set of the positions of the  lucky cars and the set of lucky spots may differ.\\

 For frustrated parking functions, it is quite easy to determine which cars/spots are lucky.

\begin{proposition}\label{prop:luckyIsSameAsAppearing}
    Let  $(a_1,a_2,\dots, a_n)$ be a frustrated parking function.  Car $j$ is lucky if and only if $a_i\neq a_j$  for all $1\leq i<j$.  Consequently, the set of lucky spots is $\{a_1,a_2,\dots, a_n\}$ and  the number of lucky cars  and spots is $|\{a_1,a_2, \dots, a_n\}|$.
\end{proposition}
\begin{proof}
    $(\Rightarrow)$ If $a_i=a_j$ for some $1\leq i<j$, then car $j$ cannot be lucky as car $i$ either parked in spot $a_j$ or spot $a_j$ was full when car  $i$ parked. In either case, car $j$ cannot park in spot $a_j$.\\

    $(\Leftarrow)$ Now suppose that $a_i\neq a_j$ for all $1\leq i<j$. We claim that spot $a_j$ is open when car $j$ starts to park, implying that car $j$ is lucky. If not, since unlucky cars park as close to spot $n$ as possible,  all the spots  from $a_j$ to $n$ must be filled. But then car $j$ cannot park, contradicting that  $(a_1,a_2,\dots, a_n)$ is  a frustrated parking function. 
\end{proof}

Before we move on, let us note that the above theorem is not true for classical parking functions. For example, considering $(1,1,2)$ under the classical parking rule, only car 1 is lucky, but as a frustrated parking function, the first and third cars are both lucky.\\

We are now ready to define our map from frustrated parking functions to Dyck paths. Let $ \mathrm{FPF}_n$ denote the set of frustrated parking functions of length $n$ and let $\mathrm{Dyck}_n$ denote the set of Dyck paths of semilength $n$. We will use $\varphi:  \mathrm{FPF}_n\rightarrow \mathrm{Dyck}_n$   to denote our map.

\begin{definition}[Frustrated Parking Function to Dyck Path Map]\label{def:FPFtoDyck}
    Given a frustrated parking function $(a_1,a_2,\dots, a_n)$, the associated Dyck path $\varphi(a_1,a_2,\dots, a_n)$, is  obtained by watching the cars park in the order they enter the lot. We start with an empty path. If the current car parking is lucky, we append an up step to our current path.  If the car is unlucky and is not the first unlucky car, we find the spot in which the previous unlucky car parked. Call this spot $s$. If the car is unlucky and it is the first unlucky car, we set $s=n+1$.  Now, suppose that our unlucky car parks in spot $t$.  Then we will append $s-t$ down steps  followed by an up step. We can think of this as adding down steps for each lucky car between the previous unlucky car and the new one. Then we add one  more down step when our unlucky car parked followed by   an up step to account for a new  car entering the parking lot. Once all cars have entered the lot, we finish by adding enough down steps so that the resulting path is a Dyck path.  Note that this last string of down steps counts the lucky cars  which do not have unlucky cars parked in an earlier spot than them.
\end{definition}

Throughout the rest of the paper, we always use $\varphi$ to denote the map from the previous definition.  Before we prove $\varphi$ is well-defined, let us illustrate our map with an example.

\begin{example}\label{ex:fpfToHDyck}
    Let us find the Dyck path for $(2,1,2,3,1,6,3,1)$. In this example, we will  be labeling the parking spots based on if the car parked there is lucky or unlucky. We use $L$ for lucky and $U$ for unlucky. \\
    \newpage
   
 The first car enters and parks in its preferred spot $2$. As it is lucky, we label with an $L$.
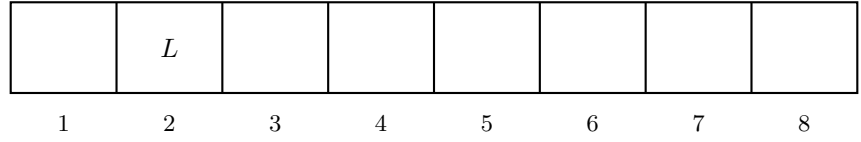
 \begin{center}
     
 \begin{tikzpicture}[scale=1]

\def\spotwidth{1.4}
\def\spotheight{1.2}

\draw[thick] (0,0) rectangle ({8*\spotwidth},\spotheight);

\foreach \i in {1,...,7} {
    \draw[thick] ({\i*\spotwidth},0) -- ({\i*\spotwidth},\spotheight);
}

\foreach \i in {1,...,8} {
    \node at ({(\i-0.5)*\spotwidth},-0.4) {\small \i};
}
\node at ({(2-0.5)*\spotwidth}, {0.5*\spotheight}) {$L$};

\end{tikzpicture}
    \end{center}
   Since we added a lucky car, we add an up step to our Dyck path to get

    
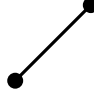
\begin{center}
        \begin{tikzpicture}[scale=1]

\draw[very thick]
(0,0) -- (1,1);

\fill (0,0) circle (3pt);
\fill (1,1) circle (3pt);

\end{tikzpicture}
    \end{center}
     
    Next, car 2 enters and parks in its preferred spot 1. Since it is lucky, we label the spot with an $L$
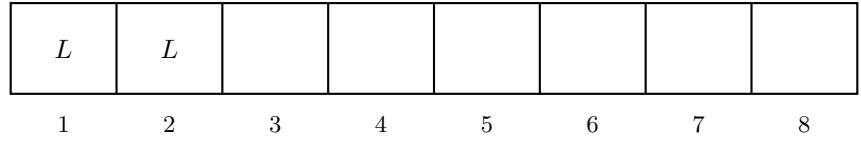
     \begin{center}
     
 \begin{tikzpicture}[scale=1]

\def\spotwidth{1.4}
\def\spotheight{1.2}

\draw[thick] (0,0) rectangle ({8*\spotwidth},\spotheight);

\foreach \i in {1,...,7} {
    \draw[thick] ({\i*\spotwidth},0) -- ({\i*\spotwidth},\spotheight);
}

\foreach \i in {1,...,8} {
    \node at ({(\i-0.5)*\spotwidth},-0.4) {\small \i};
}
\node at ({(2-0.5)*\spotwidth}, {0.5*\spotheight}) {$L$};
\node at ({(1-0.5)*\spotwidth}, {0.5*\spotheight}) {$L$};

\end{tikzpicture}
    \end{center}
  Since we added a lucky car, we add an up step to our Dyck path to get

    
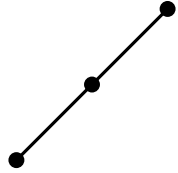
\begin{center}
        \begin{tikzpicture}[scale=1]

\draw[very thick]
(0,0) -- (1,1) -- (2,2);

\foreach \x/\y in {0/0,1/1,2/2} {
  \fill (\x,\y) circle (3pt);
}

\end{tikzpicture}
    \end{center}

Car 3 prefers spot 2, but the spot is filled.  It enters the lot and parks in spot 8. Since it is unlucky, we label the spot with a $U$.
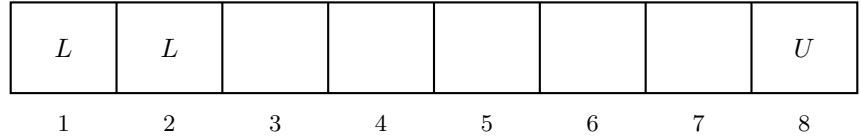
     \begin{center}
     
 \begin{tikzpicture}[scale=1]

\def\spotwidth{1.4}
\def\spotheight{1.2}

\draw[thick] (0,0) rectangle ({8*\spotwidth},\spotheight);

\foreach \i in {1,...,7} {
    \draw[thick] ({\i*\spotwidth},0) -- ({\i*\spotwidth},\spotheight);
}

\foreach \i in {1,...,8} {
    \node at ({(\i-0.5)*\spotwidth},-0.4) {\small \i};
}
\node at ({(2-0.5)*\spotwidth}, {0.5*\spotheight}) {$L$};
\node at ({(1-0.5)*\spotwidth}, {0.5*\spotheight}) {$L$};
\node at ({(8-0.5)*\spotwidth}, {0.5*\spotheight}) {$U$};
\end{tikzpicture}
    \end{center}

Since we added our first unlucky car, we set $s=8+1=9$. It parks in spot $8$, so we add $9-8=1$ down steps followed by an up step.
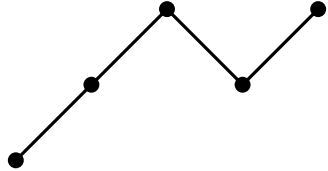
\begin{center}
\begin{tikzpicture}[scale=1]

\def\step{4}

\foreach[count=\i from 0] \x/\y in {
0/0,1/1,2/2,3/1,4/2,5/3,6/2,7/3,8/4,
9/3,10/2,11/3,12/2,13/3,14/2,15/1,16/0
}{
  \coordinate (v\i) at (\x,\y);
}

\draw[very thick]
(v0)
\foreach \i in {1,...,\step} {
  -- (v\i)
};

\foreach \i in {0,...,\step} {
  \fill (v\i) circle (3pt);
}

\end{tikzpicture}
\end{center}

The fourth car is lucky and parks in its preferred spot which is 3. Since it is lucky, we label the spot  by an $L$.

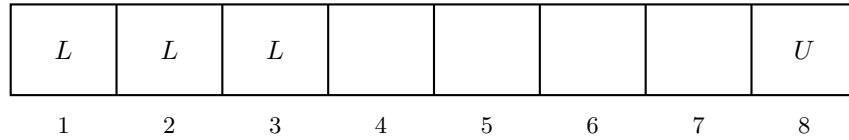
     \begin{center}
     
 \begin{tikzpicture}[scale=1]

\def\spotwidth{1.4}
\def\spotheight{1.2}

\draw[thick] (0,0) rectangle ({8*\spotwidth},\spotheight);

\foreach \i in {1,...,7} {
    \draw[thick] ({\i*\spotwidth},0) -- ({\i*\spotwidth},\spotheight);
}

\foreach \i in {1,...,8} {
    \node at ({(\i-0.5)*\spotwidth},-0.4) {\small \i};
}
\node at ({(2-0.5)*\spotwidth}, {0.5*\spotheight}) {$L$};
\node at ({(1-0.5)*\spotwidth}, {0.5*\spotheight}) {$L$};
\node at ({(8-0.5)*\spotwidth}, {0.5*\spotheight}) {$U$};
\node at ({(3-0.5)*\spotwidth}, {0.5*\spotheight}) {$L$};
\end{tikzpicture}
    \end{center}

We added a lucky car, so we add an up step to our Dyck path. 
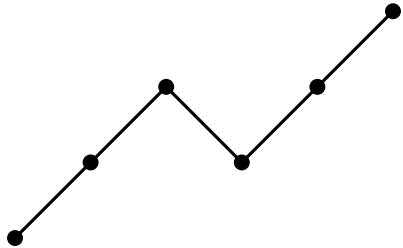
\begin{center}
\begin{tikzpicture}[scale=1]

\def\step{5}

\foreach[count=\i from 0] \x/\y in {
0/0,1/1,2/2,3/1,4/2,5/3,6/2,7/3,8/4,
9/3,10/2,11/3,12/2,13/3,14/2,15/1,16/0
}{
  \coordinate (v\i) at (\x,\y);
}

\draw[very thick]
(v0)
\foreach \i in {1,...,\step} {
  -- (v\i)
};

\foreach \i in {0,...,\step} {
  \fill (v\i) circle (3pt);
}

\end{tikzpicture}
\end{center}

Car 5 prefers spot 1, but it is filled.  So  it is unlucky and parks in spot $7$, which we label with a $U$.

\begin{center}
    
 \begin{tikzpicture}[scale=1]

\def\spotwidth{1.4}
\def\spotheight{1.2}

\draw[thick] (0,0) rectangle ({8*\spotwidth},\spotheight);

\foreach \i in {1,...,7} {
    \draw[thick] ({\i*\spotwidth},0) -- ({\i*\spotwidth},\spotheight);
}

\foreach \i in {1,...,8} {
    \node at ({(\i-0.5)*\spotwidth},-0.4) {\small \i};
}
\node at ({(2-0.5)*\spotwidth}, {0.5*\spotheight}) {$L$};
\node at ({(1-0.5)*\spotwidth}, {0.5*\spotheight}) {$L$};
\node at ({(8-0.5)*\spotwidth}, {0.5*\spotheight}) {$U$};
\node at ({(3-0.5)*\spotwidth}, {0.5*\spotheight}) {$L$};
\node at ({(7-0.5)*\spotwidth}, {0.5*\spotheight}) {$U$};
\end{tikzpicture}
    \end{center}

This unlucky car parked in spot $7$ and the previous unlucky car parked in spot $8$. So we add $8-7=1$ down steps followed by an up step.

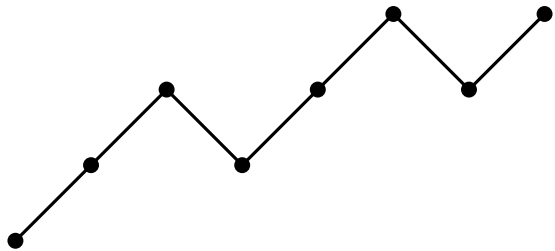
\begin{center}
\begin{tikzpicture}[scale=1]

\def\step{7}

\foreach[count=\i from 0] \x/\y in {
0/0,1/1,2/2,3/1,4/2,5/3,6/2,7/3,8/4,
9/3,10/2,11/3,12/2,13/3,14/2,15/1,16/0
}{
  \coordinate (v\i) at (\x,\y);
}

\draw[very thick]
(v0)
\foreach \i in {1,...,\step} {
  -- (v\i)
};

\foreach \i in {0,...,\step} {
  \fill (v\i) circle (3pt);
}

\end{tikzpicture}
\end{center}

Now car six enters. It is lucky and parks in its preferred spot 6. Since it is lucky, we label the spot as  $L$.
\begin{center}
    
 \begin{tikzpicture}[scale=1]

\def\spotwidth{1.4}
\def\spotheight{1.2}

\draw[thick] (0,0) rectangle ({8*\spotwidth},\spotheight);

\foreach \i in {1,...,7} {
    \draw[thick] ({\i*\spotwidth},0) -- ({\i*\spotwidth},\spotheight);
}

\foreach \i in {1,...,8} {
    \node at ({(\i-0.5)*\spotwidth},-0.4) {\small \i};
}
\node at ({(2-0.5)*\spotwidth}, {0.5*\spotheight}) {$L$};
\node at ({(1-0.5)*\spotwidth}, {0.5*\spotheight}) {$L$};
\node at ({(8-0.5)*\spotwidth}, {0.5*\spotheight}) {$U$};
\node at ({(3-0.5)*\spotwidth}, {0.5*\spotheight}) {$L$};
\node at ({(7-0.5)*\spotwidth}, {0.5*\spotheight}) {$U$};
\node at ({(6-0.5)*\spotwidth}, {0.5*\spotheight}) {$L$};
\end{tikzpicture}
    \end{center}

We added a lucky car, so we add an up step to our Dyck path.

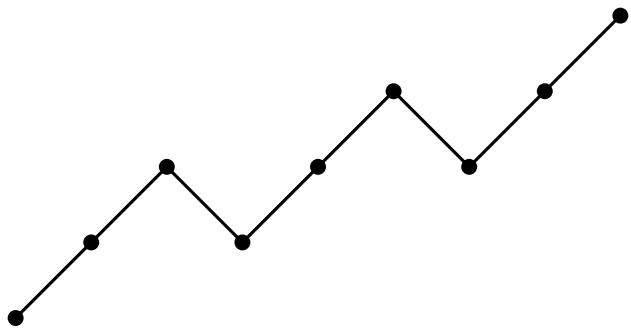
\begin{center}
\begin{tikzpicture}[scale=1]

\def\step{8}

\foreach[count=\i from 0] \x/\y in {
0/0,1/1,2/2,3/1,4/2,5/3,6/2,7/3,8/4,
9/3,10/2,11/3,12/2,13/3,14/2,15/1,16/0
}{
  \coordinate (v\i) at (\x,\y);
}

\draw[very thick]
(v0)
\foreach \i in {1,...,\step} {
  -- (v\i)
};

\foreach \i in {0,...,\step} {
  \fill (v\i) circle (3pt);
}

\end{tikzpicture}
\end{center}

Now car seven enters. It  prefers spot 3, but that spot is filled and so parks in spot $5$.  Since it is unlucky, we label spot $5$ by $U$.
\begin{center}
    
 \begin{tikzpicture}[scale=1]

\def\spotwidth{1.4}
\def\spotheight{1.2}

\draw[thick] (0,0) rectangle ({8*\spotwidth},\spotheight);

\foreach \i in {1,...,7} {
    \draw[thick] ({\i*\spotwidth},0) -- ({\i*\spotwidth},\spotheight);
}

\foreach \i in {1,...,8} {
    \node at ({(\i-0.5)*\spotwidth},-0.4) {\small \i};
}
\node at ({(2-0.5)*\spotwidth}, {0.5*\spotheight}) {$L$};
\node at ({(1-0.5)*\spotwidth}, {0.5*\spotheight}) {$L$};
\node at ({(8-0.5)*\spotwidth}, {0.5*\spotheight}) {$U$};
\node at ({(3-0.5)*\spotwidth}, {0.5*\spotheight}) {$L$};
\node at ({(7-0.5)*\spotwidth}, {0.5*\spotheight}) {$U$};
\node at ({(6-0.5)*\spotwidth}, {0.5*\spotheight}) {$L$};
\node at ({(5-0.5)*\spotwidth}, {0.5*\spotheight}) {$U$};
\end{tikzpicture}
    \end{center}
This unlucky car parked in spot $5$ and the previous unlucky car parked in spot $7$. So, we add $7-5=2$ down steps and then an up step to our  path.

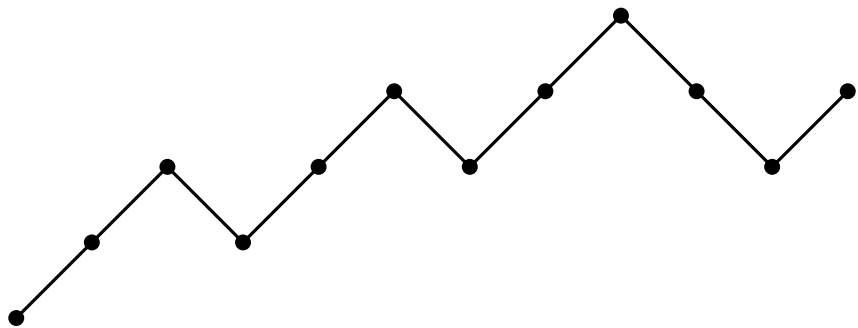
\begin{center}
\begin{tikzpicture}[scale=1]

\def\step{11}

\foreach[count=\i from 0] \x/\y in {
0/0,1/1,2/2,3/1,4/2,5/3,6/2,7/3,8/4,
9/3,10/2,11/3,12/2,13/3,14/2,15/1,16/0
}{
  \coordinate (v\i) at (\x,\y);
}

\draw[very thick]
(v0)
\foreach \i in {1,...,\step} {
  -- (v\i)
};

\foreach \i in {0,...,\step} {
  \fill (v\i) circle (3pt);
}

\end{tikzpicture}
\end{center}

Our final car enters the lot. It prefers spot 1, but this spot is filled. So it is unlucky and will park in spot 4 which we label by a $U$.

\begin{center}
    
 \begin{tikzpicture}[scale=1]

\def\spotwidth{1.4}
\def\spotheight{1.2}

\draw[thick] (0,0) rectangle ({8*\spotwidth},\spotheight);

\foreach \i in {1,...,7} {
    \draw[thick] ({\i*\spotwidth},0) -- ({\i*\spotwidth},\spotheight);
}

\foreach \i in {1,...,8} {
    \node at ({(\i-0.5)*\spotwidth},-0.4) {\small \i};
}
\node at ({(2-0.5)*\spotwidth}, {0.5*\spotheight}) {$L$};
\node at ({(1-0.5)*\spotwidth}, {0.5*\spotheight}) {$L$};
\node at ({(8-0.5)*\spotwidth}, {0.5*\spotheight}) {$U$};
\node at ({(3-0.5)*\spotwidth}, {0.5*\spotheight}) {$L$};
\node at ({(7-0.5)*\spotwidth}, {0.5*\spotheight}) {$U$};
\node at ({(6-0.5)*\spotwidth}, {0.5*\spotheight}) {$L$};
\node at ({(5-0.5)*\spotwidth}, {0.5*\spotheight}) {$U$};
\node at ({(4-0.5)*\spotwidth}, {0.5*\spotheight}) {$U$};
\end{tikzpicture}

    \end{center}
Our unlucky car parked in spot $4$ and the previous unlucky car parked in spot $5$. So, we add $5-4=1$ down steps followed by an up step to our path.

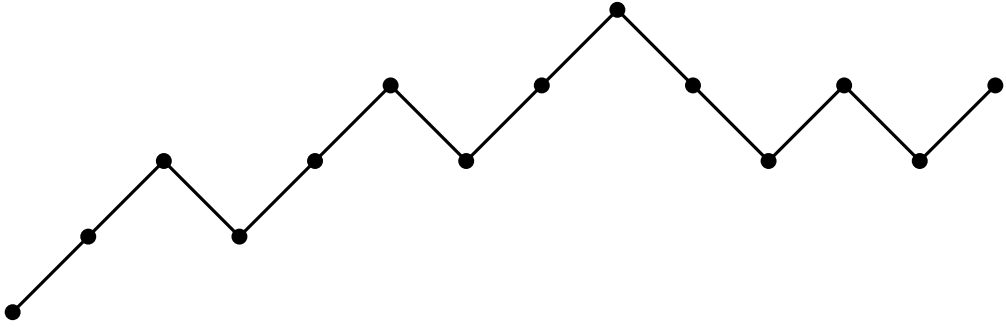
\begin{center}
\begin{tikzpicture}[scale=1]

\def\step{13}

\foreach[count=\i from 0] \x/\y in {
0/0,1/1,2/2,3/1,4/2,5/3,6/2,7/3,8/4,
9/3,10/2,11/3,12/2,13/3,14/2,15/1,16/0
}{
  \coordinate (v\i) at (\x,\y);
}

\draw[very thick]
(v0)
\foreach \i in {1,...,\step} {
  -- (v\i)
};

\foreach \i in {0,...,\step} {
  \fill (v\i) circle (3pt);
}

\end{tikzpicture}
\end{center}
We then finish the Dyck path by adding 3 more down steps. This is to make sure we get a Dyck path.  Note that this also reflects that there are no unlucky cars that parked to the left of the cars parked in spots $1,2$, and $3$.
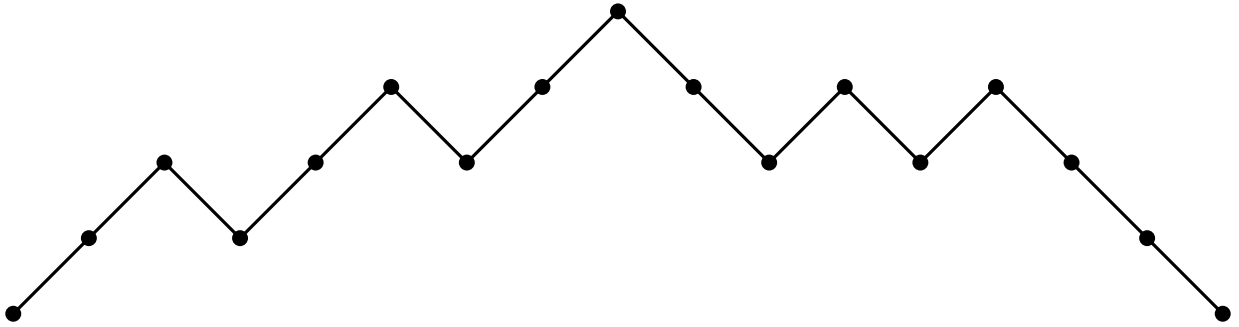
\begin{center}
\begin{tikzpicture}[scale=1]

\def\step{16}

\foreach[count=\i from 0] \x/\y in {
0/0,1/1,2/2,3/1,4/2,5/3,6/2,7/3,8/4,
9/3,10/2,11/3,12/2,13/3,14/2,15/1,16/0
}{
  \coordinate (v\i) at (\x,\y);
}

\draw[very thick]
(v0)
\foreach \i in {1,...,\step} {
  -- (v\i)
};

\foreach \i in {0,...,\step} {
  \fill (v\i) circle (3pt);
}

\end{tikzpicture}
\end{center}

\end{example}

\begin{lemma}
   The map $\varphi: \mathrm{FPF}_n\rightarrow \mathrm{Dyck}_n$ in Definition~\ref{def:FPFtoDyck} is  well-defined.
\end{lemma}
\begin{proof}
    Note that each car that enters adds both an up step and a down step. The lucky cars add an up step when they enter, and a down step when an unlucky car parks in a spot smaller than their spot or at the end of the algorithm (which counts cars that are not between two unlucky cars). For unlucky cars, we add a down step followed immediately by an up step when they park. So, we will have a total of $n$ up steps and $n$ down steps. Now we need to ensure that the path constructed by $\varphi$ never crosses below the $x$-axis.  Note that lucky cars always add an up step before a down step, so they can never be the reason why the path crosses below the $x$-axis. The unlucky cars add a down step followed immediately by an up step.  Note that for an unlucky car to park, there must be a lucky car parked in its preferred spot and that lucky car must have parked in a smaller spot than the unlucky car. So, there is always at least one lucky car that has not added a down  step yet. Thus when an unlucky car adds a down step there is always at least one up step to balance it out. We conclude that $\varphi$ is well-defined.
\end{proof}

We now introduce some definitions and notation concerning Dyck paths. Given a Dyck path $D$, an \textit{ascent} is a maximal sequence of up steps. A \textit{descent} is a maximal sequence of down steps.  A \textit{peak} occurs when we transition from an ascent to a descent. Moreover, a \textit{valley} occurs when we transition from a descent to an ascent. We use the notation $x_i$ for the length of the $i^{th}$ ascent of $D$ and $y_i$ for the length  of the $i^{th}$ descent.  For example, in the Dyck path in Figure~\ref{fig:DyckPath}, we have that

$$
    x_1 = 3 \qquad x_2=2 \qquad x_3=1 \qquad x_4=2
    $$
    and 
    $$
    y_1=3 \qquad y_2= 1 \qquad y_3=2 \qquad y_4=2
    $$

 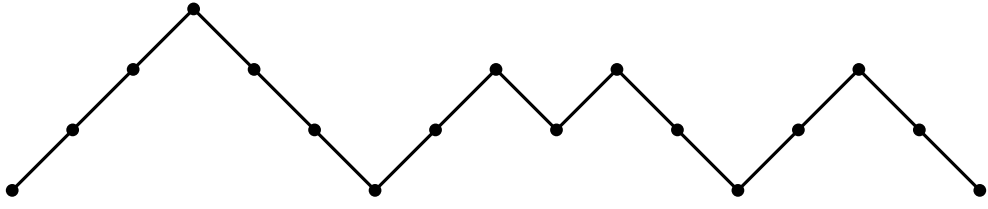
\begin{figure}
\begin{center}
\begin{tikzpicture}[scale=.8]

\def\step{16}

\foreach[count=\i from 0] \x/\y in {
0/0,1/1,2/2,3/3,4/2,5/1,6/0,
7/1,8/2,9/1,10/2,11/1,12/0,
13/1,14/2,15/1,16/0
}{
  \coordinate (v\i) at (\x,\y);
}

\draw[very thick]
(v0)
\foreach \i in {1,...,\step} {
  -- (v\i)
};

\foreach \i in {0,...,\step} {
  \fill (v\i) circle (3pt);
}

\end{tikzpicture}
\end{center}
\caption{A Dyck Path}
     \label{fig:DyckPath}
 \end{figure}

\begin{lemma}\label{lem:easyConsequencesOfDyckMap}
    Let $(a_1,a_2,\dots, a_n)\in \mathrm{FPF}_n$ and let $D=\varphi(a_1,a_2,\dots, a_n)$ (see Definition~\ref{def:FPFtoDyck}). Then we have the following.
    \begin{enumerate}
        \item[(a)] The $i^{th}$ unlucky car of $(a_1,a_2,\dots, a_n)$ is  in position 
    $$
    x_1+x_2+\cdots +x_{i}+1
    $$
    of $(a_1,a_2,\dots, a_n)$.
     \item [(b)] The number of peaks of $D$ is one more than the number of unlucky cars of $(a_1,a_2,\dots, a_n)$.
    \item [(c)] We have that
    $$
    \{a_1,a_2,\dots, a_n\} = [n]\setminus \{(n+1)-y_1, (n+1)-(y_1+y_2),\dots, (n+1)-(y_1+y_2+\cdots+ y_{k-1})\} 
    $$
    where $k$ is the number of peaks of $D$.
    \item [(d)] The number of lucky cars in $(a_1,a_2,\dots, a_n)$ is $n-k+1$ where $k$ is the number of peaks of $D$.
    \end{enumerate}
\end{lemma}
\begin{proof}
    For (a), note that when  the $i^{th}$ unlucky car enters, we end the $i^{th}$ ascent. Each car that has already entered  has added one up step to exactly one previous ascent.  So by the time the $i^{th}$ unlucky car enters, 
    $$
    x_1+x_2+\cdots + x_i
    $$
    cars have entered. The result now follows.\\

    To show (b), just note that each unlucky car ends an ascent and starts a descent. That is, it creates a peak. There is one more descent added at the  end of computing $D$ to account for any lucky car that has no unlucky cars parked in a smaller spot. Since spot 1 is always lucky, this extra descent must always be added when finding $D$. Hence, the number of peaks is the number of unlucky cars plus 1.\\
 
    To verify  (c), we can show the set of unlucky spots is 
    $$
    \{(n+1)-y_1, (n+1)-(y_1+y_2),\dots, (n+1)-(y_1+y_2+\cdots+ y_{k-1})\}
    $$ 
    When the first unlucky car parks, the number of lucky cars that parked in a larger spot is $y_1-1$ since each down step of the first descent, except for the last one, counts such lucky cars.  It follows that the first unlucky car is parked in spot $n-(y_1-1)=(n+1)-y_1$. Using this fact and that the number of lucky cars parked between the $(i-1)^{th}$ and $i^{th}$ unlucky car is $y_i-1$ gives us our desired result.\\

    Finally, note that (d) follows immediately from part (b).
\end{proof}

As we saw in Proposition~\ref{prop:weaklyDecPFAreFPFs}, the number of weakly decreasing frustrated parking functions of length $n$ is the $n^{th}$ Catalan number and hence the same as the number of Dyck paths
of semilength $n$.  Since not all frustrated parking functions are weakly decreasing, we see that the frustrated parking function to Dyck path map $\varphi$ cannot be injective. However, as we will see in the next lemma, the map is surjective.

\begin{lemma}\label{lem:DyckPathMapIsSurj}
    The map $\varphi: \mathrm{FPF}_n\rightarrow \mathrm{Dyck}_n$ in Definition~\ref{def:FPFtoDyck}  is surjective.
\end{lemma}

\begin{proof}
    By Proposition~\ref{prop:weaklyDecPFAreFPFs}, we have that the number of weakly decreasing frustrated parking functions of length $n$ is the same as the number of Dyck paths of semilength $n$. So, it suffices to show that if $(a_1,a_2,\dots, a_n)$ and $(b_1,b_2,\dots, b_n)$ are distinct weakly decreasing frustrated parking functions, then 
    $$
    \varphi(a_1,a_2,\dots, a_n)\neq \varphi(b_1,b_2,\dots, b_n)
    $$
   
    By
    Lemma~\ref{lem:easyConsequencesOfDyckMap} part (c), if $\{a_1,a_2,\dots, a_n\}\neq \{b_1,b_2,\dots, b_n\}$, then  $\varphi(a_1,a_2,\dots, a_n)\neq \varphi(b_1,b_2,\dots, b_n)$. So we can now assume that $\{a_1,a_2,\dots, a_n\}= \{b_1,b_2,\dots, b_n\}$.   Since $ (a_1,a_2,\dots, a_n)\neq(b_1,b_2,\dots, b_n)$, there exists a smallest $i$ such that $a_i\neq b_i$.
    Moreover, we claim that because  $\{a_1,a_2,\dots, a_n\}= \{b_1,b_2,\dots, b_n\}$ and   both $(a_1,a_2, \dots, a_n)$ and $(b_1,b_2,\dots, b_n)$ are weakly decreasing, it must be that in one of $(a_1,a_2,\dots, a_n)$ and $(b_1,b_2,\dots, b_n)$ the $i^{th}$ car is lucky and in the other it is unlucky. Indeed, if both were unlucky, then 
    $$
    a_i = a_{i-1}=b_{i-1}=b_i
    $$
    which is impossible since $a_i\neq b_i$. If both are lucky, then since the strings are weakly decreasing, $a_i$ and $b_i$ must both be the next smallest element of $\{a_1,a_2,\dots,a_n\}=\{b_1,b_2,\dots, b_n\}$  after $a_{i-1}=b_{i-1}$. This again implies that $a_i=b_i$, which is impossible. Now apply Lemma~\ref{lem:easyConsequencesOfDyckMap} part (a) to see that since the unlucky cars appear in different positions of the strings, it must be that $\varphi(a_1,a_2,\dots, a_n)$ and $\varphi(b_1,b_2,\dots, b_n)$ are distinct Dyck paths. 
\end{proof}

By the previous lemma, every Dyck path has a preimage. What we will now try to understand is given a Dyck path $D$, how many  frustrated parking functions map to $D$ via $\varphi$. The basic idea is that from our Dyck path, when unlucky cars enter we know which spots are lucky and which are unlucky, and hence how many lucky cars sit between unlucky cars. Note that since we need these lucky cars to sit between them, we know that by the time the $i^{th}$ unlucky car enters, every lucky car parked between the $(i-1)^{th}$ and $i^{th}$ unlucky car must have already entered the lot. This means that in the string $(a_1,a_2,\dots, a_n)$ these lucky cars must appear   in a position before the position of  the $i^{th}$ unlucky car. We also know that the $i^{th}$ unlucky car must be a repeat of a preference of a lucky car that entered before and parked in a spot earlier than where the $i^{th}$ unlucky car parks. Taking advantage of these facts, we can completely determine how many frustrated parking functions map to a fixed Dyck path $D$.  Before we prove this, let us provide a  concrete example to illustrate the idea.

\begin{example}
    Let's consider the Dyck path $D$ found in Figure~\ref{fig:DyckPath}. We see that there are 8 up steps and 8 down steps, telling us that any frustrated parking function that maps to $D$ has length $8$.  We also see that there are 4 ascents and 4 descents implying there are $4-1=3$ unlucky cars. Using  $x_i$ to denote the length of the $i^{th}$ ascent and $y_i$ for the length of the $i^{th}$  descent, we have the following:
    $$
    x_1 = 3 \qquad x_2=2 \qquad x_3=1 \qquad x_4=2
    $$
    and 
    $$
    y_1=3 \qquad y_2= 1 \qquad y_3=2 \qquad y_4=2
    $$
    Lemma~\ref{lem:easyConsequencesOfDyckMap} part (a)  tells us that the unlucky cars are in positions $x_1+1=4, x_1+x_2+1=6,$ and  $x_1+x_2+x_3 + 1=7$.  Using Lemma~\ref{lem:easyConsequencesOfDyckMap} part (c), we see that
    \begin{align*}
        \{a_1,a_2,\dots, a_n\} &= [8]\setminus \{9-3, 9-4,9-6\}\\
        &=[8]\setminus \{6,5,3\}\\
        &=\{1,2,4, 7,8\}
    \end{align*}
    So our string $(a_1,a_2,\dots, a_8)$ consists of the characters $1,2,4,7$ and $8$. Consequently, the unlucky spots are $3,5,6$ and are filled in descending order.\\

    We are going to build our string $(a_1,a_2,\dots, a_8)$ by going through each descent. Descent one corresponds to the first unlucky car entering the lot. As we noted earlier, this car enters in position 4 and parks in spot $6$.  This tells us that the cars that prefer spots 7 and 8 must enter before the fourth car. So in the first 3 characters of $(a_1,a_2,\dots, a_8)$ we must see $7$ and $8$. The number of ways to place these two preferences is $3\cdot 2$.  We also need to decide the preference of the first unlucky car. It must be a preference that has already appeared in the string. Of the three positions before this unlucky car, only one has a preference that is less than $6$. While we don't currently know what preference is in this position of the string, we know its value must be less than $6$. So whatever it ends up being, we can assign the first unlucky car's preference to be this one.  It follows that there is 1 choice for this unlucky car's preference. So for descent one, there are $3\cdot 2\cdot  1$ choices that we could have possibly made.\\

    Now we move to descent two which corresponds to the second unlucky car. Recall that this car enters in position $6$ of $(a_1,a_2,\dots, a_8)$.  Moreover, this unlucky car parks in spot $5$. Since the previous unlucky car parked in spot $6$, there are no lucky cars that park between the first and second unlucky car. So, we can just determine what the preference of the second unlucky car is. There are a total of 5 cars that have entered by the time this unlucky car enters. Of them one is unlucky and two of them have preferences that are too big (namely the ones with preference 7 and 8).  So there are $5-3=2$ possible choices for what the second unlucky car's preference could be.  It follows for descent two, there are $2$ choices that we could have possibly made.\\

    Next, we move to descent three. This corresponds with the third unlucky car which is the seventh car to enter.   As we saw earlier it parks in spot $3$ and so there is one lucky car that parks between the second and third unlucky cars. This lucky car  must have entered before the third unlucky car. In other words, it must have entered in the first 6 positions. Of these 6 positions, $2$ of them are unlucky cars and $2$ of them are lucky cars that we already placed. So there are $6-4=2$ possible choices for the position of this lucky car in  $(a_1,a_2,\dots, a_8)$. Now we need to decide what the third unlucky car's preference is. It must be a preference of a previous car. Of the six positions, two are unlucky and three of them are lucky cars whose preference is too big (namely $4,7,8$). So there are $6-5=1$ possible choices for the preference of this unlucky car. It follows for descent three, there are $2\cdot 1$ choices that we could have possibly made.\\

    Finally we move to descent four. This does not correspond to an unlucky car, but rather to placing the remaining lucky cars that have not already been placed. Note that there are two preferences that have not been placed (namely $1$ and $2$) corresponding to the fact that there are no unlucky cars parked in earlier spots than $1$ and $2$. Currently there are 2 open spots in the string $(a_1,a_2,\dots, a_8)$ and we must place two characters into them. It follows for descent four, there are $2\cdot 1$ choices that we could have possibly made.\\

    Since our choices for each descent are independent, the total number of frustrated parking functions   that map to $D$ via $\varphi$ is 
    $$
    (3\cdot 2\cdot 1) (2)( 2\cdot 1)(2\cdot 1)=48
    $$
 \end{example}

In the previous example, we saw that there were 48 different frustrated parking functions that map to the Dyck path in Figure~\ref{fig:DyckPath}. It would be nice if there was a connection between that Dyck path and the number 48. Fortunately, such an association exists.

\begin{definition}
    Let $D$ be a Dyck path.  The  \textit{height} of a down step in $D$  is the  $y$-value where the down step begins. A \textit{height labeled Dyck path} is a Dyck path together with a labeling of the down steps by positive integers such that every label is at most the height of the down step.
\end{definition}

If we look at the Dyck path in Figure~\ref{fig:DyckPath}, we see that the sequence of heights along the down steps is given by
$$
3,2,1,2,2,1,2,1
$$
It then follows that the number of height labeled Dyck paths of $D$ is $3\cdot 2\cdot 1\cdot 2\cdot 2\cdot 1\cdot 2\cdot 1
=48$, which is exactly the number of frustrated parking functions which map to $D$. As we will soon see, this is no coincidence.\\

In the  proof of the following theorem, we will use the notation $(n)_m$ to denote the falling factorial. So $(n)_m= n(n-1)\cdots (n-(m-1))$.

\begin{proposition}\label{prop:numFPFisSameAsHL}
    Let $D$ be a Dyck path. The number of frustrated parking functions $(a_1,a_2,\dots, a_n)$ such that $\varphi(a_1,a_2,\dots, a_n)=D$ 
    is the number of height labeled Dyck paths with underlying Dyck path $D$.
\end{proposition}
\begin{proof}
    Fix a Dyck path $D$ and let's suppose $D$ has $k$ descents. We explain how to construct a frustrated parking function $(a_1,a_2,\dots, a_n)$ with the property that $\varphi(a_1,a_2,\dots, a_n)=D$.  By Lemma~\ref{lem:easyConsequencesOfDyckMap}, $D$ completely determines the number of unlucky cars in $(a_1,a_2,\dots, a_n)$ as well as the positions of the unlucky cars in $(a_1,a_2,\dots, a_n)$. Moreover, we can tell which characters appear in $(a_1,a_2,\dots,a_n)$ and consequently, we can determine which spots are unlucky.  Note that for any frustrated parking function, the unlucky spots get filled in reverse (i.e.~the largest unlucky spot gets filled first and so on).  If a lucky car parks in a spot larger thanwhere an unlucky car parks, by the time this unlucky car enters the lot, that lucky car must have already parked. It follows that this lucky spot's value must appear in the string before the unlucky car's position.  This implies that to find a frustrated parking function which maps to $D$, we can focus on when the unlucky cars enter the lot.\\

Suppose that  we are considering when the $i^{th}$ unlucky car is entering. By Lemma~\ref{lem:easyConsequencesOfDyckMap} part (a), this unlucky car is in position $x_1+x_2+\cdots+x_i+1$ of $(a_1,a_2,\dots, a_n)$.  There are $y_i-1$ lucky cars who's preference needs to be put into the string up to this point. Currently, the number of open positions is the total number of positions before $x_1+x_2+\cdots+x_i+1$, which is $x_1+x_2+\cdots+x_i$ minus the number of positions that have been filled.  The number of the positions that have been filled is $y_1+y_2+\cdots+y_{i-1}$. So the total number of open positions where  we can place these $y_i-1$ preferences is 
$$
h_i:=(x_1-y_1) +(x_2-y_2)+\cdots +(x_{i-1}-y_{i-1}) +x_i
$$
It follows that the number of ways to do this is 
$$
(h_i)_{y_i-1} 
$$

Then we need to determine preferences for that $i^{th}$ unlucky car. We must choose a preference that is smaller than the position that this unlucky car parks in.  These preferences have not yet been put into our string, but we can instead assign a position in the string that is not currently filled and then give the $i^{th}$ unlucky car this preference once it is filled. Currently the number of unfilled positions is 
$$
(x_1-y_1) +(x_2-y_2)+\cdots +(x_{i-1}-y_{i-1}) +x_i-(y_i-1)
$$
So, the total number of ways to fill in the string at this step is 
$$
(h_i)_{y_i-1} ((x_1-y_1) +(x_2-y_2)+\cdots +(x_{i-1}-y_{i-1}) +x_i-(y_i-1)) = (h_i)_{y_i}
$$

Note that $h_i$ is the height of the first down step of the $i^{th}$ descent. It follows that 
$$
(h_i)_{y_i}
$$
is the product of the heights along the $i^{th}$ descent.\\

We now need to place the preferences $1,2,\dots,y_k$ into our string. This corresponds to the $y_k$ lucky cars that have no unlucky cars parked in an earlier spot. Moreover, the number of empty spots is 
$$
h_k :=(x_1-y_1) +(x_2-y_2)+\cdots +(x_{k-1}-y_{k-1}) +x_k
$$
So there are 
$$
(h_k)_{y_k}
$$
ways to do this. Note that this is the product of the heights along the last descent. \\

To finish, let us  note that all of our choices are made independently, so the total number of possible strings that map to $D$ is 
$$
(h_1)_{y_1}(h_2)_{y_2}\cdots(h_k)_{y_k}
$$
which is exactly the product of the heights along the down steps of $D$. The result now follows.
\end{proof}

As shown in~\cite{callan2009combinatorialsurveyidentitiesdouble}, the number of height labeled Dyck paths of semilength $n$ is $(2n-1)!!$ \footnote{In Callan's article, the up steps are labeled as opposed to down steps. Reversing the labeled Dyck paths provides a bijection between the up step labeled and down step labeled variants.}. Hence Proposition~\ref{prop:numFPFisSameAsHL} gives us our main theorem.

\begin{theorem}\label{thm:numTotalFPFsIsDoubleFact}
    The number of frustrated parking functions of length $n$ is $(2n-1)!!$.
\end{theorem}

In fact, we can say more. The map $\varphi: \mathrm{FPF}_n\rightarrow 
\mathrm{Dyck}_n$ maps frustrated parking functions with $n-k$ unlucky cars (and hence $k$ lucky cars) to Dyck paths with $n-k+1$ peaks.  It is known (see for example~\cite{callan2009combinatorialsurveyidentitiesdouble}) the number of height labeled Dyck paths with $n-k+1$ peaks is
$$
\SecondEulerian{n}{n+1-(n-k+1)} = \SecondEulerian{n}{k}
$$
where 
$$
\SecondEulerian{n}{k}
$$
is the second order Eulerian number indexed by $n$ and $k$. These numbers count, among other things, descents in Stirling permutations (as opposed to descents in regular permutations which are counted by the regular Eulerian numbers). See the OEIS entry~\href{https://oeis.org/A008517}{A008517}~(\cite{oeisA002538}) for more information.

\begin{corollary}\label{cor:numWithKDescentsIsSecondEuler}
    The number of frustrated parking functions of length $n$ with $k$ lucky cars    is 
    $$
   \SecondEulerian{n}{k}
    $$
    This is also the number of frustrated parking functions of length $n$ with $k$ lucky spots.
\end{corollary}

 \begin{remark}\label{rem:avgLuckyCars}
     Bona~\cite{bonaStirlingPerms}   showed that the average number of descents  in a Stirling permutation of length $2n$ is $\dfrac{2n+1}{3}$. Combining this with Corollary~\ref{cor:numWithKDescentsIsSecondEuler}, shows that the average number of lucky cars/spots in a frustrated parking function of length $n$ is also $\dfrac{2n+1}{3}$. 
\end{remark}

\subsection{Counting Lucky Cars and Lucky Spots}\label{sec:countLucky}
Previously we used lucky and unlucky cars to define our map from frustrated parking functions to Dyck paths. Here we will use this map to take a deeper look at lucky cars and lucky spots.\\

Let $T$ be  a subset of $[n]$. We will use the notation $LS^n_T$ to denote the set of frustrated parking functions of length $n$ whose set of lucky spots is $T$. Similarly, we will use $LC^n_T$ to denote the set of frustrated parking functions of length $n$ whose set of positions of lucky cars is $T$. For example, when $n=3$, we have that 
$$
LS^3_{\{1,2\}} =\{(1,1,2),(1,2,1), (1,2,2), ( 2,1,1), (2,1,2), (2,2,1)\}
$$
and 
$$
LC^3_{\{1,2\}} =\{(1,2,1), (1,2,2), (1,3,1), (2,1,1), (2,1,2), (3,1,1)\}
$$

The reader may have noticed that while the sets are not the same, they do have the same size. Clearly the number of frustrated parking functions that have $k$ lucky cars and $k$ lucky spots are the same. However, we can say something stronger. We will show that for all $n$ and all $T\subseteq [n]$
$$
|LS^n_T| = |LC^n_T|
$$
Before we do this, let us note that the analogue for classical parking functions is not true. For example, for $n=3$, there are five classical parking functions with set of lucky spots $\{1,2\}$ namely 
$$
(1,2,1), (1,2,2), (2,1,1), (2,1,2), (2,2,1)
$$
and six with set of  positions of lucky cars $\{1,2\}$ namely
$$
(1,2,1), (1,2,2), (1,3,1), (2,1,1), (2,1,2), (3,1,1)
$$

\begin{lemma}\label{lem:luckyCarsandSpotsInDyckPath}
    Let $(a_1,a_2,\dots, a_n)$ be a frustrated parking function and let $D=\varphi(a_1,a_2,\dots, a_n)$. Then we have the following.
    \begin{enumerate}
        \item[(a)] The $i^{th}$ car is lucky if and only if the $i^{th}$ up step of $D$ is not part of a valley.
    \item [(b)]  The $i^{th}$ spot is lucky if and only if the $(n-i+1)^{th}$ down step of $D$ is not part of a valley. 
    \end{enumerate}
\end{lemma}

\begin{proof}
    Part (a) follows from the fact that each lucky car provides an up step to $D$. Moreover, all up steps except those immediately following a descent come from lucky cars. As for part (b), the lucky spots are counted backwards from $n$ for each down step except for the last down step of the descent i.e.~a down step that is part of a valley.
\end{proof}

\begin{theorem}\label{thm:numLuckCarsSameAsNumLuckySpotsForFixedSet}
 Let $n\geq 1$. If $T\subseteq [n]$, then 
 $$
 |LS^n_T|=|LC^n_T|
 $$
 In other words, for  fixed $n$, the number of frustrated parking functions of length $n$ with set of lucky spots $T$ is the same as those whose set of positions of lucky cars is $T$.
\end{theorem}
\begin{proof}
    Given a Dyck path $D$, let $D^{rev}$ be the Dyck path obtained by reversing the string $D$ and replacing each up step with a down step and vice versa. Clearly this map is a bijection on Dyck paths. Moreover, it is straightforward to see that the number of height labeled Dyck paths with underlying Dyck path $D$ is the same as the number of height labeled Dyck paths with underlying Dyck path $D^{rev}$. Thus, Proposition~\ref{prop:numFPFisSameAsHL} implies that 
    $$
    |\varphi^{-1}(D)| = |\varphi^{-1}(D^{rev})| 
    $$  
    So, for each Dyck path $D$, there is a bijection between $\varphi^{-1}(D)$ and $\varphi^{-1}(D^{rev})$. When we compare $D$ to $D^{rev}$, we have that the $i^{th}$ up step of $D$ is the $(n-i+1)^{th}$ down step of $D^{rev}$ and vice versa. So by Lemma~\ref{lem:luckyCarsandSpotsInDyckPath}, we get a bijection that swaps the set of positions of lucky cars with the set of lucky spots. The result now follows.
\end{proof}

 We now consider the case when  only the first $k$ cars/spots are lucky.  In the next proposition and throughout the remainder of the article, we use $S(n,k)$ to denote the Stirling number of the second kind indexed by $n$ and $k$. So $S(n,k)$ is the  number of set partitions of $[n]$ with $k$ blocks. Note that $k!S(n,k)$ is the number of surjections from $[n]$ to $[k]$.

\begin{proposition}\label{prop:numFirstSpotsLuckyisk!S(n,k)}
    The number of frustrated parking functions of length $n$ whose set of positions of lucky  cars as well as the set of frustrated parking functions whose lucky spots is $[k]$ is $k!S(n,k)$. That is 
    $$
    |LC^n_{[k]}| = |LS^n_{[k]}| =k!S(n,k)
    $$
\end{proposition}

\begin{proof}

    We prove the result for the set of lucky spots $[k]$. The result for the set of positions of lucky cars will then follow from Theorem~\ref{thm:numLuckCarsSameAsNumLuckySpotsForFixedSet}.\\

    By Proposition~\ref{prop:luckyIsSameAsAppearing},  $LS^n_{[k]}$ is  the set of frustrated parking functions with preference  set $[k]$. By Corollary~\ref{cor:ifLuckyCarIsIntervalItWorks} any string of whose characters set is $[k]$ is a frustrated parking function.  Combining these facts shows that  $LS^n_{[k]}$  is the set of strings of length $n$ whose character set is $[k]$ such that each element of $[k]$ is used at least once. This set of strings is in bijection with the set of surjections from $[n]$ to $[k]$ from which the result now follows.
\end{proof}

The previous proposition implies that  the  number of frustrated parking functions such that once a car/spot is unlucky, all later cars/spots are unlucky is given by
$$
\sum_{k=1}^n k!S(n,k) 
$$
 These numbers are known as the Fubini numbers (and also as the ordered Bell numbers). They appear in the OEIS as A000670.

\begin{corollary}\label{cor:numLuckyThenUnluckyIsFubini}
The number of frustrated parking functions of length $n$ where  once a spot is unlucky all later spots are unlucky is     the $n^{th}$ Fubini number. Moreover, 
the number of frustrated parking functions of length $n$ where once a car is unlucky, the remaining cars are unlucky is the $n^{th}$ Fubini number.
\end{corollary}

 \begin{remark}

A \textit{unit interval parking function} is a sequence $(a_1,a_2,\dots, a_n)$ in which, under the classical parking rule, car $i$ parks in spot $a_i$ or $a_i+1$. Originally shown in~\cite{Hadaway2021} (and  later  in ~\cite{unitIntRFubini} and~\cite{unitIntPFs}), the number of unit interval parking functions of length $n$ are counted by the $n^{th}$ Fubini number.  Hence frustrated parking functions where once a car/spot is unlucky all remaining cars/spots are unlucky are in bijection with unit interval parking functions. We should note that while all these sets are subsets of parking functions, they are not the same. For example, $(1,1,3)$ is a unit interval parking function, but it is not  a frustrated parking function.   The set of frustrated parking functions where once a spot is unlucky all later spots are unlucky is given by the disjoint union
$$
\bigcup_{k=1}^n LS^n_{[k]}
$$
As was explained in the proof of Proposition~\ref{prop:numFirstSpotsLuckyisk!S(n,k)}, each $LS^n_{[k]}$ is the set of strings of length $n$ where every character in $[k]$ appears at least once.  Thus, 
$$
\bigcup_{k=1}^n LS^n_{[k]}
$$
is the set of strings of length $n$ over $[n]$ such that if $j$ appears in the string, and $i\leq j$, then $i$ appears. Such strings are often referred to as \textit{Cayley permutations}. In~\cite[Lemma 4.2]{invInParkingFunc}, an explicit bijection is given between Cayley permutations and unit interval parking functions. This gives an explicit  bijection between unit interval parking functions and frustrated parking functions of length $n$ where once a spot is unlucky, all later spots are unlucky.  Let us also note that by Corollary~\ref{cor:numLuckyThenUnluckyIsFubini}, there are the same number of unit interval parking functions and  frustrated parking functions where once a car is unlucky all remaining cars are unlucky. However, for $n\geq 3$, this set of frustrated parking functions are not Cayley permutations. It would be interesting to find an explicit bijection between these frustrated parking functions and unit interval parking functions.
 \end{remark}

 Using the definition of $\varphi$, we can see that if $(a_1,a_2,\dots, a_n)\in LC^n_{[k]}$ for some $k$, then $\varphi(a_1,a_2,\dots, a_n)$ cannot contain the consecutive subsequence $DUU$. This is because $DUU$ in the Dyck path is the exact situation when an unlucky car parked, and then a lucky car entered.  From this we can use our map to prove the following.
\begin{corollary}
    The number of Dyck paths that do not contain the consecutive subsequence $DUU$  is $2^{n-1}$ and the number of height labeled Dyck paths whose underlying Dyck path does not contain the consecutive subsequence $DUU$  is the $n^{th}$ Fubini number. The same result holds if we replace $DUU$ with $UDD$.
\end{corollary}
\begin{proof}
First note that the claim in the last sentence follows from the previous claim since there is a bijection between (height labeled) Dyck paths that do not have $DUU$ and those that do not have $UDD$.\\

    By the proof of Lemma~\ref{lem:DyckPathMapIsSurj}, restricting the map $\varphi:\mathrm{FPF}_n\rightarrow \mathrm{Dyck}_n$ to weakly decreasing frustrated parking functions gives a bijection. As mentioned earlier, the frustrated parking functions where once a car is unlucky, all remaining cars are unlucky map to Dyck paths that do not contain the sequence $DUU$. So, it suffices to find the number of weakly decreasing frustrated parking functions that once a car is unlucky all subsequent cars are unlucky. Since spot 1 is always lucky, and the sequence is weakly decreasing, the only repeated character in the string can be 1. So, we just need to decide what other characters appear in the string and then write them in decreasing order, followed by a string of 1's. If we want $k$ non-1 characters, there are
    $$
    {n-1\choose k}
    $$
    choices. Since $k$ can range from 0 to $n-1$, the total is
    $$
    \sum_{k=0}^{n-1}     {n-1\choose k}
 = 2^{n-1}
    $$
Hence the total number of Dyck paths that do not contain $DUU$ is $2^{n-1}$. To finish the proof, apply the result in Corollary~\ref{cor:numLuckyThenUnluckyIsFubini}.
\end{proof}

\section{A Generalization Using Graphs}\label{sec:graphGeneralization}
Though we have framed frustrated parking functions as a situation where the drivers get frustrated and drive as far from their preferred spot as possible, our initial framework came from graphs. Because this viewpoint allows for further generalization, we wish to share this interpretation here. \\

Let $G$ be a graph with vertex set $[n]$, and suppose that cars may travel along its edges and park at its vertices. We have $n$ cars, each with a preferred parking vertex. The cars arrive one at a time. A car starts at its preferred vertex. If the vertex is unoccupied, it parks there. Otherwise, from its current vertex $v$, it looks among the neighbors of $v$ whose labels are larger than $v$. If one or more are empty, the car moves to the largest empty neighbor and parks there. If none are empty, the car drives to the smallest labeled neighbor of $v$ whose label is larger than $v$ and repeats the same procedure. This process continues until the car either parks or reaches a vertex with no larger neighbors, at which point it is unable to park. We say that a sequence  $(a_1,a_2,\dots, a_n)$ of preferences is a  \textit{largest neighbor parking function of $G$} if all cars with these preferences can park on $G$ according to our rule. \\

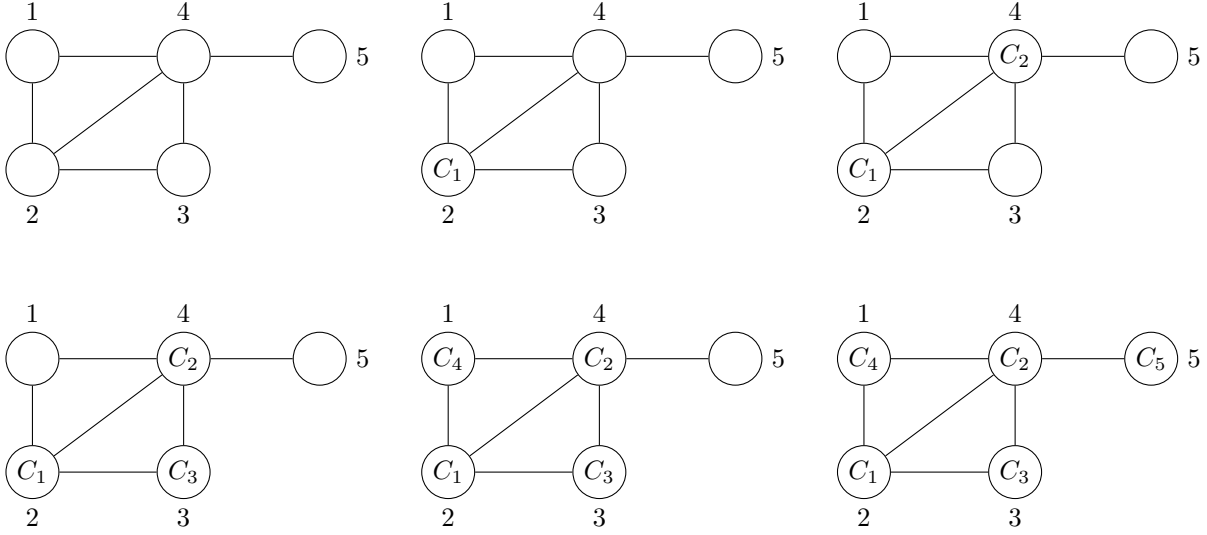
\begin{figure}
    \centering
\begin{tikzpicture}[
    vertex/.style={
        circle,
        draw,
        fill=white,
        minimum size=7mm,
        inner sep=0pt
    }
]

\node[vertex, label=above:$1$] (1) at (0,1.5) {};
\node[vertex, label=below:$2$] (2) at (0,0) {};
\node[vertex, label=below:$3$] (3) at (2,0) {};
\node[vertex, label=above:$4$] (4) at (2,1.5) {};
\node[vertex, label=right:$5$] (5) at (3.8,1.5) {};

\draw (1)--(2);
\draw (2)--(3);
\draw (3)--(4);
\draw (4)--(1);
\draw (2)--(4);
\draw (4)--(5);

\begin{scope}[shift={(5.5,0)}]
\node[vertex, label=above:$1$] (1) at (0,1.5) {};
\node[vertex, label=below:$2$] (2) at (0,0) {$C_1$};
\node[vertex, label=below:$3$] (3) at (2,0) {};
\node[vertex, label=above:$4$] (4) at (2,1.5) {};
\node[vertex, label=right:$5$] (5) at (3.8,1.5) {};

\draw (1)--(2);
\draw (2)--(3);
\draw (3)--(4);
\draw (4)--(1);
\draw (2)--(4);
\draw (4)--(5);

\end{scope}

\begin{scope}[shift={(11,0)}]
\node[vertex, label=above:$1$] (1) at (0,1.5) {};
\node[vertex, label=below:$2$] (2) at (0,0) {$C_1$};
\node[vertex, label=below:$3$] (3) at (2,0) {};
\node[vertex, label=above:$4$] (4) at (2,1.5) {$C_2$};
\node[vertex, label=right:$5$] (5) at (3.8,1.5) {};

\draw (1)--(2);
\draw (2)--(3);
\draw (3)--(4);
\draw (4)--(1);
\draw (2)--(4);
\draw (4)--(5);

\end{scope}

\begin{scope}[shift={(0,-4)}]
\node[vertex, label=above:$1$] (1) at (0,1.5) {};
\node[vertex, label=below:$2$] (2) at (0,0) {$C_1$};
\node[vertex, label=below:$3$] (3) at (2,0) {$C_3$};
\node[vertex, label=above:$4$] (4) at (2,1.5) {$C_2$};
\node[vertex, label=right:$5$] (5) at (3.8,1.5) {};

\draw (1)--(2);
\draw (2)--(3);
\draw (3)--(4);
\draw (4)--(1);
\draw (2)--(4);
\draw (4)--(5);

\end{scope}

\begin{scope}[shift={(5.5,-4)}]
\node[vertex, label=above:$1$] (1) at (0,1.5) {$C_4$};
\node[vertex, label=below:$2$] (2) at (0,0) {$C_1$};
\node[vertex, label=below:$3$] (3) at (2,0) {$C_3$};
\node[vertex, label=above:$4$] (4) at (2,1.5) {$C_2$};
\node[vertex, label=right:$5$] (5) at (3.8,1.5) {};

\draw (1)--(2);
\draw (2)--(3);
\draw (3)--(4);
\draw (4)--(1);
\draw (2)--(4);
\draw (4)--(5);

\end{scope}

\begin{scope}[shift={(11,-4)}]
\node[vertex, label=above:$1$] (1) at (0,1.5) {$C_4$};
\node[vertex, label=below:$2$] (2) at (0,0) {$C_1$};
\node[vertex, label=below:$3$] (3) at (2,0) {$C_3$};
\node[vertex, label=above:$4$] (4) at (2,1.5) {$C_2$};
\node[vertex, label=right:$5$] (5) at (3.8,1.5) {$C_5$};

\draw (1)--(2);
\draw (2)--(3);
\draw (3)--(4);
\draw (4)--(1);
\draw (2)--(4);
\draw (4)--(5);

\end{scope}
\end{tikzpicture}
    \caption{A graph together with how cars park with preference string $(2,2,3,1,1)$}
    \label{fig:graphExample}
\end{figure}

We now consider an example.
\begin{example}
    Say $G$ is the graph 
    shown in Figure~\ref{fig:graphExample}.   Let's consider the preference sequence $(2,2,3,1,1)$. The first car is placed on vertex 2. Since it is empty, it parks there. The next car is placed at vertex $2$. Since it is occupied, it looks to see if there are any larger neighbors that are empty.  Vertices 3 and 4 are both larger and empty. We pick the largest, which is 4.  So car $2$ parks in spot $4$. Next the third car is placed at vertex 3. It is empty so it parks there. Now the fourth car is placed at vertex 1. It is empty so it parks there. Finally the fifth car is placed at vertex 1. It is full. So it looks for neighbors that are empty. There are none. So it goes to the smallest neighbor larger than vertex 1. This is vertex 2. So it moves to vertex 2. It is full, so it looks for empty larger neighbors. There are none. So, it moves to the smallest vertex larger than 2, this is vertex 3.  Vertex 3 is full. All of its larger neighbors are full, so it moves to the smallest neighbor that is larger than 3. That is vertex 4. It is full so it looks for a larger empty neighbor. There is one (and only one), so it parks at vertex 5.
\end{example}

We now collect some families of largest neighbor parking functions.
    \begin{itemize}
        \item Let $P_n$ be the path graph on $[n]$ with edge set $i(i+1)$ for $1\leq i\leq n-1$. Then the largest neighbor parking functions on $P_n$  are the classical parking functions and hence there are $(n+1)^{n-1}$ such sequences.
        \item Let $K_n$ be the complete graph on $[n]$. Then the largest neighbor parking functions  on $K_n$ are the frustrated parking functions and hence there are $(2n-1)!!$ such sequences.
        \item Let $E_n$ be the empty graph on $[n]$ (i.e.~the graph with no edges). Then the largest neighbor parking functions on $E_n$ are  the permutations of $[n]$ and hence there are $n!$ such sequences.
        \item Let $ST_n$ be the star graph with central vertex $1$ (i.e.~the graph with edges $1i$ for $2\leq i\leq n$). There are $\dfrac{(n+1)!}{2}$ largest neighbor parking functions on $ST_n$.
        \item Let $\overline{ST}_n$ be the star on $[n]$ with central vertex $n$ (i.e.~the graph with edges $in$ for $1\leq i\leq n-1$).  There are $\dfrac{(n+1)!}{2}$ largest neighbor parking functions on $\overline{ST}_n$.
    \end{itemize}
We wish to point out that while $ST_n$ and $\overline{ST}_n$ are isomorphic and have the same number of largest neighbor parking functions, the sets of these parking functions are different. For example $\underbrace{(1,1,\ldots,1)}_{n\text{ }}$ is valid for $ST_n$, but not valid for $\overline{ST}_n$ when $n\geq 3$. At first glance, this might suggest that isomorphic graphs have the same number of largest neighbor parking functions. However, $P_3$ is isomorphic to both $ST_3$ and $\overline{ST}_3$, but has a different number of largest neighbor parking functions.  This is all to say that the labeling absolutely matters when considering largest neighbor parking functions.\\

We now provide a few basic facts about largest neighbor parking functions.

\begin{proposition}\label{prop:largestNeighborImpliesClassical}
    Let $G$ be a graph on $[n]$. Every  largest neighbor parking function of $G$ is a (classical) parking function  of length $n$.  
\end{proposition}
\begin{proof}
    We prove the contrapositive. Suppose that $(a_1,a_2,\dots, a_n)$ is not a parking function. Then there is some $i$ such that there are more than $n-i+1$ cars with a preference at least $i$. But then all these cars will need to park at a vertex $i$ or greater, which is impossible.
\end{proof}

Given two words $(b_1,b_2,\dots, b_m)$ and $c=(c_1,c_2,\dots, c_n)$ a \textit{shuffle} of  $(b_1,b_2,\dots, b_m)$ and $c=(c_1,c_2,\dots, c_n)$ is a word  of length $m+n$ such that both $(b_1,b_2,\dots, b_m)$ and $(c_1,c_2,\dots, c_n)$ are substrings of the shuffle. For example, the shuffles of $(1,1)$ and $(2,4,3)$ are
\begin{align*}
   (1,1,2,4,3),\quad
(1,2,1,4,3),\quad
&(1,2,4,1,3),\quad
(1,2,4,3,1),\quad
(2,1,1,4,3),\quad
(2,1,4,1,3),\quad
(2,1,4,3,1),\\
&(2,4,1,1,3),\quad
(2,4,1,3,1),\quad
(2,4,3,1,1) 
\end{align*}

 Now suppose we have a graph $G$ on $[n]$ with  connected components $C_1,C_2,\dots, C_k$.  Since  the $C_i$'s are disjoint, if we have a string $(a_1,a_2,\dots,a_n)$ over $[n]$ we can identify which entries belong to  each $C_i$. Thus, we can extract substrings consisting of only elements of $C_i$. If $(a_1,a_2,\dots,a_n)$ is a largest neighbor parking function for $G$, then the substring for $C_i$ must be a largest neighbor parking function for $C_i$.  Moreover, if we take a largest neighbor parking function for each $C_i$ and shuffle them together, we get a largest neighbor parking function on $G$.    Thus, we have the following proposition. Note that this proposition implies  when considering largest neighbor parking functions, we may assume the graph is connected.

\begin{proposition}\label{prop:neighborShuffle}
    Let $G$ be a graph on $[n]$ with connected components $C_1, C_2,\dots, C_k$. A  sequence   is a largest neighbor parking function of $G$ if and only if it is a shuffle of largest neighbor parking functions of $C_1,C_2,\dots, C_k$.
\end{proposition}

We finish this section by considering weakly decreasing largest neighbor parking functions.

\begin{proposition}\label{prop:neighborWeaklyDec}
    Let $G$ be a graph on $[n]$ that contains $P_n$ (i.e.~the path graph  with edge set $i(i+1)$ for $1\leq i\leq n-1$).  Then every weakly decreasing parking function is a largest neighbor parking function on $G$. Consequently the number of weakly decreasing largest neighbor parking functions on $G$ is the $n^{th}$ Catalan number.
\end{proposition}
\begin{proof}
By Proposition~\ref{prop:largestNeighborImpliesClassical}, it is enough to show that every weakly decreasing parking function can park on $G$. Let $(a_1,a_2,\dots, a_n)$ be a parking functions such that $a_1\geq a_2\geq\cdots\geq a_n$. Then for all $i$, $a_i\leq n+1-i$. Thus, when the $i^{th}$ car is placed on its preferred vertex, either the vertex is unoccupied or there is at least one vertex $j$ with $a_i<j$ that is unoccupied. Since $G$ contains $P_n$, either the car will park before it gets to vertex $j$ or it will eventually get to vertex $j$. Since $j$ is unoccupied, it will park there. In any case, the car can park. 
\end{proof}

We should note that you cannot drop the assumption of the graph containing $P_n$ from the previous proposition. For example,  let $G$ be the  the graph on $[3]$ with edges $13$ and $23$ (i.e.~$G=\overline{ST}_3$), the sequence $(1,1,1)$ is a weakly decreasing parking function, but cannot park on $G$.\\

We hope this new way of thinking about parking functions in terms of parking on graphs inspires the reader to explore their favorite families of graphs and find the largest neighbor parking functions for those graphs.\\

\section{Open Questions}
In Corollary~\ref{cor:numWithKDescentsIsSecondEuler}, we saw that the number of frustrated parking functions of length $n$ with $k$ lucky cars/spots is the second order Eulerian number $\displaystyle\SecondEulerian{n}{k}$. This was shown using height labeled Dyck paths. However, a better known interpretation for the   second order Eulerian number is that it counts descents in \textit{Stirling permutations}. These are words of length $2n$ such that each character in $[n]$ appears exactly twice and such that for all $i\in [n]$, any character that appears between the two values of $i$ is larger than $i$.   
\begin{question}
    Can one find a nice bijection that takes frustrated parking functions with $k$ lucky cars/spots to Stirling permutations with $k$ descents?
\end{question}

As is the case with classical parking functions, we have the typical questions concerning ascents, descents, and plateaus.

\begin{question}
    What does the distribution of descent/ascents/plateaus in frustrated parking functions look like?  
\end{question}

By Proposition~\ref{prop:neighborShuffle}, if we take the disjoint union of $G$ and $H$, the largest neighbor parking functions of $G\cup H$ arise from shuffling the largest neighbor parking functions of $G$ and $H$. However, there are many other ways to combine graphs such as gluing along subgraphs or taking the product of the graphs. This leads to the following question.

\begin{question}
  How do different ways of combining graphs affect their largest neighbor parking functions?
\end{question}

In Proposition~\ref{prop:neighborWeaklyDec}, we showed that as long as $G$ contains the path graph, the weakly decreasing largest neighbor parking functions are the same as the weakly decreasing classical parking functions. But what about other classes of parking functions? For example, we could count them based on if they are weakly increasing, or how many lucky spots/cars they have. 
\begin{question}
    Given a graph $G$, what do the counts of subfamilies of largest neighbor parking functions look like?
\end{question}
Another natural subfamily one could look at are those that park either at their preferred vertex or one that is adjacent to it. This is a generalization of unit interval parking functions in the classical sense. Note that in this setup, for the complete graph $K_n$, every largest neighbor parking function is a unit interval parking function. With this framework, we see that  frustrated parking functions are exactly the unit interval parking functions for the complete graph.\\

\section*{Acknowledgments}
This research was initiated through the Summer Undergraduate Research Program (SURP) at Loyola Marymount University, whose support made this work possible. Some of the results in this article  appeared in the undergraduate theses of the second and third authors. The second author  was also supported by Pi Mu Epsilon to present this work at the Joint Mathematics Meeting 2026. None of the original mathematics in this paper was created using AI. We did make use of   AI  to help make some of the TikZ diagrams that appear in this article as well as for general (non-mathematical) proofreading.

 \bibliographystyle{alpha}
\bibliography{ref}

\newcommand{\etalchar}[1]{$^{#1}$}
\begin{thebibliography}{CHMM{\etalchar{+}}24}

\bibitem[B\'09]{bonaStirlingPerms}
Mikl\'os B\'ona.
\newblock Real zeros and normal distribution for statistics on {S}tirling permutations defined by {G}essel and {S}tanley.
\newblock {\em SIAM J. Discrete Math.}, 23(1):401--406, 2008/09.

\bibitem[Bau19]{Baumgardner2019}
Alyson Baumgardner.
\newblock The naples parking function.
\newblock Honors Contract--Graph Theory, Florida Gulf Coast University, 2019.

\bibitem[BEH{\etalchar{+}}24]{unitIntRFubini}
S.~Alex Bradt, Jennifer Elder, Pamela~E. Harris, Gordon~Rojas Kirby, Eva Reutercrona, Yuxuan Wang, and Juliet Whidden.
\newblock Unit interval parking functions and the {$r$}-{F}ubini numbers.
\newblock {\em Matematica}, 3(1):370--384, 2024.

\bibitem[Cal09]{callan2009combinatorialsurveyidentitiesdouble}
David Callan.
\newblock A combinatorial survey of identities for the double factorial.
\newblock {\em https://arxiv.org/abs/0906.1317}, 2009.

\bibitem[CCH{\etalchar{+}}21]{chooseOwnAdv}
Joshua Carlson, Alex Christensen, Pamela~E. Harris, Zakiya Jones, and Andr\'es Ramos~Rodr\'iguez.
\newblock Parking functions: choose your own adventure.
\newblock {\em College Math. J.}, 52(4):254--264, 2021.

\bibitem[CDMY21]{intervalPFs}
Emma Colaric, Ryan DeMuse, Jeremy~L. Martin, and Mei Yin.
\newblock Interval parking functions.
\newblock {\em Adv. in Appl. Math.}, 123:Paper No. 102129, 17, 2021.

\bibitem[CEH{\etalchar{+}}25]{invInParkingFunc}
Kyle Celano, Jennifer Elder, Kimberly~P. Hadaway, Pamela~E. Harris, Amanda Priestley, and Gabe Udell.
\newblock Inversions in parking functions.
\newblock {\em https://arxiv.org/abs/2508.11587}, 2025.

\bibitem[CHJ{\etalchar{+}}20]{GenOfParkingFuncAllowBack}
Alex Christensen, Pamela~E. Harris, Zakiya Jones, Marissa Loving, Andr\'es Ramos~Rodr\'iguez, Joseph Rennie, and Gordon~Rojas Kirby.
\newblock A generalization of parking functions allowing backward movement.
\newblock {\em Electron. J. Combin.}, 27(1):Paper No. 1.33, 18, 2020.

\bibitem[CHMM{\etalchar{+}}24]{permutationInvariantParkingFunctions}
Douglas~M. Chen, Pamela~E. Harris, J.~Carlos Mart\'inez~Mori, Eric~J. Pab\'on-Cancel, and Gabriel Sargent.
\newblock Permutation invariant parking assortments.
\newblock {\em Enumer. Comb. Appl.}, 4(1):Paper No. S2R4, 25, 2024.

\bibitem[CMHJ{\etalchar{+}}25]{unitIntPFs}
Lucas Chaves~Meyles, Pamela~E. Harris, Richter Jordaan, Gordon Rojas~Kirby, Sam Sehayek, and Ethan Spingarn.
\newblock Unit-interval parking functions and the permutohedron.
\newblock {\em J. Comb.}, 16(3):281--301, 2025.

\bibitem[Had21]{Hadaway2021}
Kimberly~P. Hadaway.
\newblock On combinatorial problems of generalized parking functions.
\newblock BA Thesis, Williams College, 2021.

\bibitem[KW66]{konheim1966occupancy}
Alan~G Konheim and Benjamin Weiss.
\newblock An occupancy discipline and applications.
\newblock {\em SIAM Journal on Applied Mathematics}, 14(6):1266--1274, 1966.

\bibitem[MM24]{mori}
J.~Carlos Mart\'inez~Mori.
\newblock What is{$\ldots$} a parking function?
\newblock {\em Notices Amer. Math. Soc.}, 71(8):1062--1065, 2024.

\bibitem[{OEI}26]{oeisA002538}
{OEIS Foundation Inc.}
\newblock The {O}n-{L}ine {E}ncyclopedia of {I}nteger {S}equences.
\newblock \url{https://oeis.org/A002538}, 2026.
\newblock Entry A002538: Second-order Eulerian numbers.

\bibitem[Yan15]{yan}
Catherine~H. Yan.
\newblock Parking functions.
\newblock In {\em Handbook of enumerative combinatorics}, Discrete Math. Appl. (Boca Raton), pages 835--893. CRC Press, Boca Raton, FL, 2015.

\end{thebibliography}

 \end{document}